\documentclass[preprint,12pt]{amsart} %leqno is the option to put formula numbers on the left side
\usepackage{amssymb}
\usepackage{graphicx}
\theoremstyle{plain} %text of this environment is typesetted in italics
\newtheorem{theo}{\indent\sc Theorem}[section]
\newtheorem{lemm}[theo]{\indent\sc Lemma}
\newtheorem{cor}[theo]{\indent\sc Corollary}
\newtheorem{prop}[theo]{\indent\sc Proposition}
\theoremstyle{definition} %text of this environment is typesetted in roman letters
\newtheorem{defi}[theo]{\indent\sc Definition}
\newtheorem{rem}[theo]{\indent\sc Remark}

\newcommand{\vna}{von Neumann algebra}

\newcommand{\sps}{standard probability space}

\newcommand{\Hs}{Hilbert space}
\newcommand{\sfplalg}{subfactor planar algebra}
\newcommand{\unplal}{unshaded planar algebra}

\newcommand{\alev}{almost everywhere}

\newcommand{\wrt}{with respect to}
\newcommand{\vesp}{vector space}

\newcommand{\R}{\mathbb R}
\newcommand{\C}{\mathbb C}
\newcommand{\N}{\mathbb N}
\newcommand{\h}{\mathcal H}
\newcommand{\lN}{{\ell^2(\mathbb N)}}
\newcommand{\lNN}{{\lN\otimes\lN}}
\newcommand{\deltafrac}[1]{\delta^{\frac{#1}{2}}}

\newcommand{\deltamoin}[1]{\delta^{-\frac{#1}{2}}}
\newcommand{\Vertvert}[1]{\Vert #1\Vert}
\newcommand{\de}{\delta}
\newcommand{\dmm}{\delta^{-2}}
\newcommand{\dmi}{\delta^{-1}}
\newcommand{\dr}{\delta^{-\frac{r}{2}}}
\newcommand{\dkr}{\delta^{-\frac{k+r}{2}}}
\newcommand{\sh}{ s+s^*}

\newcommand{\U}{\frac{\cup-1}{\deltafrac{1}}}
\newcommand{\cupbul}[1]{\cup^{\bullet #1}}
\newcommand{\2}{{[-2;2]}}
\newcommand{\Ln}{{L^2(\2,\nu)}}
\newcommand{\co}{\mathcal C(\2)}
\newcommand{\MP}{M_\Pl}

\newcommand{\LMP}{L^2(M_\Pl)}
\newcommand{\LM}{L^2(M)}
\newcommand{\AM}{A\subset M}
\newcommand{\LA}{L^2(A)}
\newcommand{\Pl}{\mathcal P}
\newcommand{\wPl}{{Gr(\Pl)}}
\newcommand{\AbA}{{_A{\overline b}_A}}
\newcommand{\ALAA}{{_AL^2(A)_A}}
\newcommand{\APA}{{_A{\overline{\Pl_1}}_A}}
\newcommand{\AVA}{{_A\overline{V}_A}}
\newcommand{\NMA}{N_M(A)}

\newcommand{\intI}{\int^\oplus_Y}
\newcommand{\Li}{{L^\infty(Y,\nu)}}
\newcommand{\Ld}{L^2(Y,\nu)}
\newcommand{\Bh}{\mathcal B(\h)}
\newcommand{\Bo}{\mathcal B}
\newcommand{\YDn}{(Y,\mathcal D,\nu)}
\newcommand{\Grk}{Gr_k(\Pl)}
\newcommand{\LMk}{L^2(M_k)}

\begin{document}

\title[Unshaded planar algebras and their associated II$_1$ factors]{Unshaded planar algebras and their associated II$_1$ factors}

\author[A. Brothier]{Arnaud Brothier$^*$}

%%%%%%%%%%%%%%% footnote %%%%%%%%%%%%%%%%

\thanks{ 
$^{*}$Partially supported by the Region Ile de France and by ERC Starting Grant VNALG-200749.
}
%%%%%%%%%%%% Authors addresses %%%%%%%%%%%%%
\address{% First Author
Institute of Mathematics of Jussieu 
Universtity Paris Diderot 
175 rue du Chevaleret 
Paris 75013 
France
Phone number : +33(0)601768274}
\email{brot@math.jussieu.fr}

\address{% First Author
KU Leuven
Department of Mathematics 
Celestijnenlaan 200B - Bus 2400
3001 Heverlee
Belgium}
\email{arnaud.brothier@wis.kuleuven.be}

%%%%%%%%%%%%%%%%%%%%%%%%%%%%%%%%%%%%%%%%%

\begin{abstract}
Guionnet et al. gave a construction of a II$_1$ factor associated to a \sfplalg.
In this paper we define an \unplal.
To any \unplal\ $\Pl$ we associate a finite \vna\ $\MP$.
We prove that $\MP$ is a II$_1$ factor that contains a generic maximal abelian subalgebra called the cup subalgebra.
\end{abstract}

\maketitle
\keywords{ 
Planar algebras; von Neumann algebras; maximal abelian subalgebras; free probability.
\section*{Introduction}

%Historic
The theory of subfactors has been initiated by Jones \cite{Jones_index_for_subfactors}; he introduced the notion of planar algebras \cite{jones_planar_algebra} in order to describe the standard invariant, (see also Peters \cite{peters_planar_haagerup_graph} for an introduction to planar algebras).
Popa \cite{popa_system_construction_subfactor} proved that any standard invariant comes from a subfactor, Popa and Shlyakhtenko proved \cite{Popa_Shlyakhtenko_univ_prop_LF_subfactor} that the subfactor can be realized in the free group factor $L(\mathbb{F}_\infty)$.
Guionnet et al. \cite{GJS_random_matrices_free_proba_planar_algebra_and_subfactor, JSW_orthogonal_approach_planar_algebra,GJS_semifinite_algebra} gave a planar proof of this result.
To any finite depth \sfplalg\ $\Pl$, they associate a subfactor $L(\mathbb F_s)\subset L(\mathbb F_t)$, where $L(\mathbb F_s)$ and $L(\mathbb F_t)$ are interpolated free group factors (see \cite{Dykema_LF_t,Radulescu_Random_matrices_LF_t}).
Furthermore, they give an explicit formula for the parameters $s$ and $t$.
The formula is a linear combination of the Jones index and the global index of the subfactor planar algebra $\Pl$.
Note that those results have been partially proved independently by Sunder and Kodiyalam in \cite{sunder_constructionGJS,Sunder_surlaconstructionGJSII}.

%Aim and Main result
In \cite{JSW_orthogonal_approach_planar_algebra}, Jones et al. associate a tower of II$_1$ factors $\{M_k,\,k\geqslant 0\}$ to any subfactor planar algebra $\Pl$.
We consider the first \vna\ $M_0$ that appears in this tower.
We call it the \textit{\vna\ associated to $\Pl$} and denote it by $\MP$.
The aim of this paper is to give an extension of this construction for a class of planar algebras that we called \textit{unshaded}, see definition \ref{defi_unshaded_planar_algebra}.
We consider a generic abelian subalgebra of the \vna\ associated to the planar algebra. 
We call it \textit{the cup subalgebra}.
The main result of this paper is the following:

\begin{theo}
To any \unplal\ $\Pl$ is associated a \vna\ $\MP$.
This \vna\ $\MP$ is a II$_1$ factor and the cup subalgebra is a maximal abelian subalgebra.
\end{theo}

%Motivations
The introduction of the notion of \unplal\ has been motivated by the following example:
Consider the space of non commutative complex polynomials in $l$ variables with monomials of even degree $\C_{\text{even}}\langle X_1,\cdots,X_l\rangle$.
It has a planar algebra structure by following the tensor rules, see Jones \cite[Example 2.6]{jones_planar_algebra}.
We want to extend this planar algebra structure to all polynomials $\C\langle X_1,\cdots,X_l\rangle$ (not necessarily of even degree).
In order to do that we need a larger set of planar tangles.
We define the collection of \textit{unshaded planar tangles} that can have an odd or even number of strings coming from a disk.
If there is an odd number of boundary points on a disk, there is no hope in trying to shade it.
For this reason, we call them unshaded planar tangles.
We give here a definition similar to the one given by Peters \cite{peters_planar_haagerup_graph} for the shaded case:

\begin{defi}
An unshaded planar tangle has an outer disk, a finite number of inner disks, and a finite number of non-intersecting strings.
A string can be either a closed loop or an edge with endpoints on boundary circles.
We require that there be a marked point (denoted by $\star$) on the boundary of each disk, and that the inner disks are ordered.
\end{defi}

Here is an example of an unshaded planar tangle:
$$\includegraphics[width=4cm]{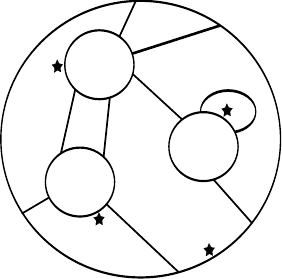}.$$
We will suppose that two tangles are equal if they are isotopic.
We can compose them by placing a tangle inside an interior disk of another lining up the marked points, and connecting endpoints of strands.
By considering the action of the planar tangles on the subfactor planar algebra $\C_{\text{even}}\langle X_1,\cdots,X_l\rangle$, one can define in an obvious way the action of the class of unshaded planar tangles on the algebra of all polynomials $\C\langle X_1,\cdots,X_l\rangle$.

We define a multiplication $\star$ on $\C\langle X_1,\cdots,X_l\rangle$ which is:

\begin{align*}
X_{i_1}\cdots X_{i_n}\star X_{j_1}\cdots X_{j_m} &= X_{i_1}\cdots X_{i_n}X_{j_1}\cdots X_{j_m} \\
&+ \delta_{i_n,j_1} X_{i_1}\cdots X_{i_{n-1}}X_{j_2}\cdots X_{j_m}\\
 &+ \delta_{i_n,j_1} \delta_{i_{n-1},j_2} X_{i_1}\cdots X_{i_{n-2}}X_{j_3}\cdots X_{j_m} + \cdots,
\end{align*}
where $\delta_{ij}$ is the Kronecker symbol.
For example,
$$X_1X_2X_3\star X_3X_2=X_1X_2X_3^2X_2+X_1X_2^2+X_1.$$

Be aware that $\star$ is used in two different contexts, for a multiplication and to indicate the marked points of a tangle.

Following the construction, of Guionnet et al. we provide a \vna\ which is isomorphic to the free group factor $L(\mathbb{F}_l)$ with $l$ generators, and $\{X_1,\cdots,X_l\}$ is a free semicircular family in the sense of Voiculescu \cite{Voiculescu_dykema_nica_Free_random_variables}.
Consider a finite dimensional real \Hs\ $\h$ with an orthonormal basis $\{h_1,\cdots,h_l\}$.
In \cite[Theorem 2.6.2.]{Voiculescu_dykema_nica_Free_random_variables}, Voiculescu defined a map $s$ from $\h$ to the bounded operators of the full Fock space on the complexification of $\h$.
The \vna\ generated by the $s(h)$ is isomorphic to the free group factor $L(\mathbb F_l)$. Furthermore the $s(h_i)$ form a free semicircular family.
The map $X_i\longmapsto s(h_i)$ defines an isomorphism of \vna s.
Hence, the construction of Guionnet et al. can be seen as an extension of the construction of Voiculescu.

This work has also been motivated by a question about maximal abelian subalgebras (MASAs).
Consider the free group factor generated by $a_1,\cdots,a_l$. 
Consider the \textit{radial MASA} generated by the element $\sum_i(a_i+a_i^{-1})$ and the \textit{generator MASA} generated by $a_1$.
It is conjectured that the generator and the radial MASAs are not isomorphic.
We call two MASAs $\AM$ and $B\subset N$ \textit{isomorphic} if there exists an isomorphism of \vna s $\phi:M\longrightarrow N$ such that $\phi(A)=B$.
The cup subalgebra of $\MP$ where $\Pl$ is the algebra of polynomials $\C\langle X_1,\cdots,X_l\rangle$ is of the same nature as the radial MASA.
We prove that it shares many property with the radial MASA.
However we have been unable to prove that the cup subalgebra, the radial MASA and the generator MASA are pairwise isomorphic or distinct.

%Plan
Here is a precise description of the paper.\\
In section \ref{section generalization de GJSW}, we give a definition of an unshaded planar algebra $\Pl$.
It is a countable family of finite dimensional complex vector spaces $\{\Pl_n,\,n\geqslant 0\}$ on which the set of unshaded planar tangles is acting and following a few extra axioms.
The algebra of non commutative polynomials is an example of an unshaded planar algebra.
Furthermore, the planar algebra $\tilde\Pl=\{\Pl_{2n},\ n>0\}$ is a \sfplalg.
Conversely, any self-dual \sfplalg\ is of this form.
We recall that a \sfplalg\ is \textit{self-dual} if its principal graph is equal to its dual graph.

We associate to an \unplal\ $\Pl$ a \vna\ $\MP$.

The principal difficulty is to show that the multiplication is bounded, see proposition \ref{prop_GJSW_mult_cont}.
This is done by a graphical proof.
We consider the trace of the unshaded planar algebra that can be extended as a faithful normal state $tr$ on $M_\Pl$.
In particular $M_\Pl$ is a finite \vna.
We denote the product of this \vna\, by $\star$ in reference to the construction.

If $\Pl$ is a subfactor planar algebra, then it is proved in \cite{JSW_orthogonal_approach_planar_algebra} that $M_\Pl$ is a II$_1$ factor.
The rest of this article is devoted to proving that $M_\Pl$ is a II$_1$ factor for any unshaded planar algebra $\Pl$.

To do this, we proceed as in \cite{JSW_orthogonal_approach_planar_algebra}, we look at the cup subalgebra $A\subset M_\Pl$ introduced in section  \ref{section_GJSW_cup}.
The cup subalgebra $\AM_\Pl$ is an abelian von Neumann subalgebra generated by a self-adjoint element of $\Pl_2$.
We show that the $A$-bimodule $\LMP\ominus\LA$ is isomorphic to an infinite direct sum of coarse correspondences. 
The main difficulty is to show that the $A$-bimodule generated by $\Pl_1$ is isomorphic to a direct sum of coarse correspondences.
The proof uses perturbation theory of operators.
Once we know the $A$-bimodule structure of $\LMP$ we get in particular that the cup subalgebra is a MASA and $\MP$ is a II$_1$ factor.

In the appendix, we associate to an \unplal\ $\Pl=\{\Pl_n,\,n\geqslant 0\}$ a tower of \vna s $\{M_k,\ k>0\}$.
We show that each $M_k$ is a II$_1$ factor.
Furthermore, we prove that the \sfplalg\ associated to the subfactor $M_0\subset M_1$ is the self-dual planar algebra $\tilde\Pl=\{\Pl_{2n},\ n>0\}$.

In the last part of the appendix we discuss MASAs and invariants for them.

%%%%%%%%%%%%%%%%%%%%%%%%%%%%%%%%%%%%%%%%%%%%%%%%%%%%%%%%%%%%%%%%%%%%%%%%%%%%%%%%%%%%%%%%%%%%%%%%%%%%%%%%%%%%%%%%%%%%%%%%%%%%%%%%%%%%%%%%%%%%%%%%%

\section{A \vna\ associated to a \unplal}\label{section generalization de GJSW}

%%%%%%%%%%%%%%%%%%%%%%%%%%%%%%%%%%%%%%%%%%%%%%%%%%%%%%%%%%%%%%%%%%%%%%%%%%%%%%%%%%%%%%%%%%%%%%%%%%%%%%%%%%%%%%%%%%%%%%%%%%%%%%%%%%%%%%%%%%%%%%%%%

\subsection{Definition of an unshaded planar algebra}\label{subsection_GJSW_definition_unshaded_planar_algebra}

%%%%%%%%%%%%%%%%%%%%%%%%%%%%%%%%%%%%%%%%%%%%%%%%%%%%%%%%%%%%%%%%%%%%%%%%%%%%%%%%%%%%%%%%%%%%%%%%%%%%%%%%%%%%%%%%%%%%%%%%%%%%%%%%%%%%%%%%%%%%%%%%%

\begin{defi}\label{defi_unshaded_planar_algebra}
An unshaded planar algebra $\Pl$ is a family of finite dimensional complex vector spaces $\{\Pl_n\}_{n\geqslant 0}$, called the $n$-box spaces.
We suppose that the dimension of $\Pl_0$ is equal to $1$.
For any $n\geqslant 0$, there is an anti-linear involution $*:\Pl_n\longrightarrow\Pl_n$.
Any unshaded planar tangle defines a linear map $\Pl_{n_1}\otimes ...\otimes \Pl_{n_k}\rightarrow \Pl_{n_0}.$
The natural number $n_i$ corresponds to the number of endpoints on the $i$th interior disk and $n_0$ to number of endpoints on the exterior disk.
The action of the collection of unshaded planar tangle is compatible with the composition of tangles.

We require that $\Pl$ is spherically invariant and the action of the tangles is compatible with the anti-linear involutions $*$ of the $\Pl_n$,
i.e.
$$T(b_{n_1},\ldots,b_{n_i})^*=T^*({b_{n_1}}^*,...,{b_{n_i}}^*)$$
for any vectors $b_{n_i}\in \Pl_{n_i}$ and any tangle $T$, where $T^*$ is the reflection of $T$ for any line in the plane.
We denote the modulus of the planar algebra $\Pl$ by $\delta$, which is the value of a closed loop
$$\includegraphics[width=3cm]{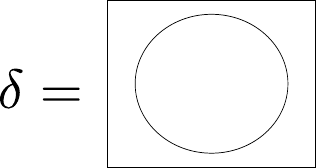}$$
and suppose that $\delta > 1$.
We assume that $\Pl$ is non degenerate, i.e. for any $n\geqslant 0$, the sesquilinear form $\langle\cdot,\cdot\rangle$ defined on each $\Pl_n$ by
$$\includegraphics[width=4cm]{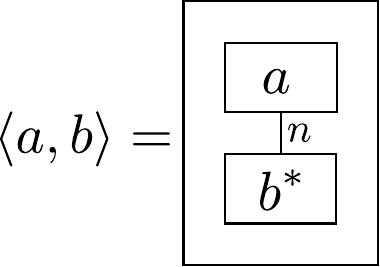}$$
is an inner product of $\Pl_n$.
\end{defi}

In all the paper, a planar algebra will denote an unshaded planar algebra.

%%%%%%%%%%%%%%%%%%%%%%%%%%%%%%%%%%%%%%%%%%%%%%%%%%%%%%%%%%%%%%%%%%%%%%%%%%%%%%%%%%%%%%%%%%%%%%%%%%%%%%%%%%%%%%%%%%%%%%%%%%%%%%%%%%%%%%%%%%%%%%%%%

\subsection{Setup}\label{subsection_GJSW_notations}

%%%%%%%%%%%%%%%%%%%%%%%%%%%%%%%%%%%%%%%%%%%%%%%%%%%%%%%%%%%%%%%%%%%%%%%%%%%%%%%%%%%%%%%%%%%%%%%%%%%%%%%%%%%%%%%%%%%%%%%%%%%%%%%%%%%%%%%%%%%%%%%%%

We follow the setup of \cite{JSW_orthogonal_approach_planar_algebra}.
Let $\Pl=(\Pl_n)_{n\geqslant 0}$ be an unshaded planar algebra.
Let $Gr(\Pl)$ be the graded vector space equal to the algebraic direct sum of the vector spaces $\Pl_n$,
i.e. $Gr(\Pl)=\bigoplus_{n\geqslant 0}\Pl_n$.
We extend the inner product of each $\Pl_n$ on $Gr(\Pl)$ making it an orthogonal direct sum.
We still write $\Pl_n$ when it is considered as the $n$-graded part of $Gr(\Pl)$.
To simplify the pictures, as in the article of Kodiyalam and Sunder \cite{sunder_constructionGJS} we decorate strands in a planar tangle with non-negative integers to represent cabling of that strand.
For example :
$$\includegraphics[width=1.5cm]{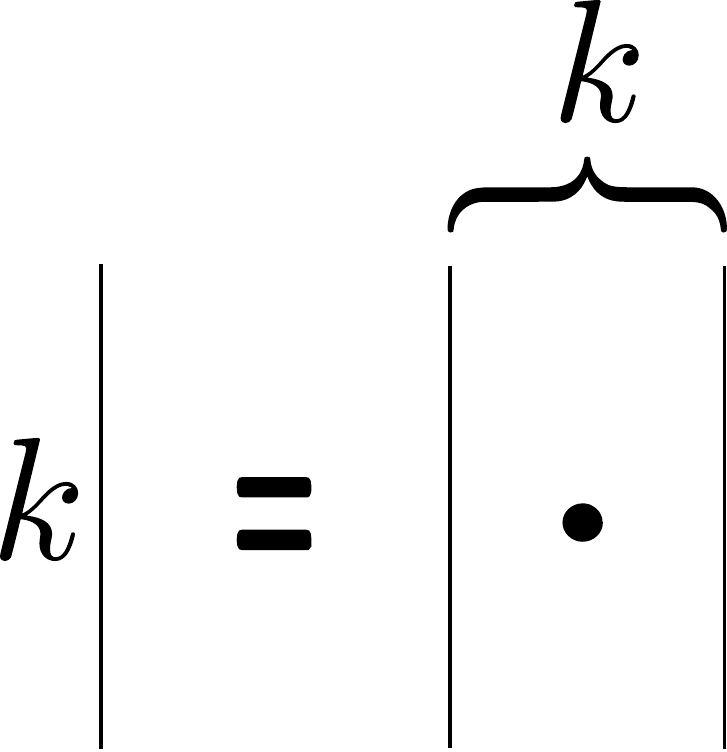}.$$
An element $a\in\Pl_n$ will be represented as a box:
$$\includegraphics[width=1.5 cm]{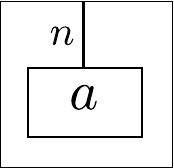}.$$
We assume that the marked point is at the top left of the box.
If not we will denote this marked point by $\star$.
We define a multiplication on $\wPl$ by requiring that if $a\in\Pl_n$ and $b\in\Pl_m$, then $a\bullet b\in\Pl_{n+m}$ is given by
$$\includegraphics[width=4cm]{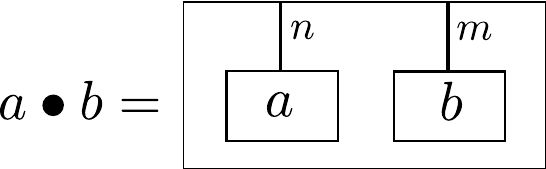}.$$
Consider the element of $\Pl_2$:
$$\includegraphics[width=1.7cm]{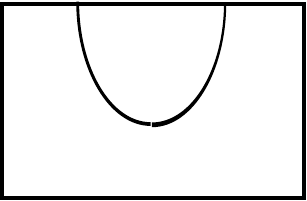}.$$
We call it cup and denote it by the symbol $\cup$.
The element
$$\includegraphics[width=2cm]{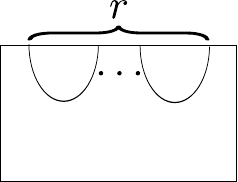}$$
is denoted by the symbol $\cupbul r.$
For example, $\cupbul 2=\cup\bullet\cup$.
We use the convention that $0=\cupbul{k}$ for $k<0$ and $1=\cupbul{0}$.
In particular, $a\bullet\cupbul r=0$ if $r\leqslant -1$, for any $a$.

%%%%%%%%%%%%%%%%%%%%%%%%%%%%%%%%%%%%%%%%%%%%%%%%%%%%%%%%%%%%%%%%%%%%%%%%%%%%%%%%%%%%%%%%%%%%%%%%%%%%%%%%%%%%%%%%%%%%%%%%%%%%%%%%%%%%%%%%%%%%%%%%%

\subsection{Construction of a von Neumann algebra}

%%%%%%%%%%%%%%%%%%%%%%%%%%%%%%%%%%%%%%%%%%%%%%%%%%%%%%%%%%%%%%%%%%%%%%%%%%%%%%%%%%%%%%%%%%%%%%%%%%%%%%%%%%%%%%%%%%%%%%%%%%%%%%%%%%%%%%%%%%%%%%%%%

We follow the construction given in \cite{JSW_orthogonal_approach_planar_algebra}.

We equip the graded vector space $\wPl$ with the product $\star$ described by the following planar tangle:
$$\includegraphics[width=5.5cm]{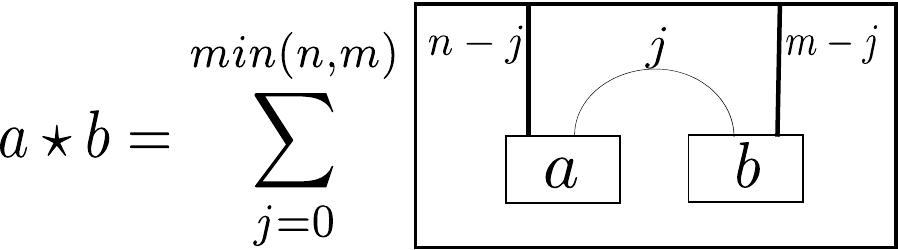},$$
where $a\in \Pl_n$ and $b\in\Pl_m$.
The $*$-structure on $\wPl$ is the involution coming from the planar algebra.

\begin{prop}
The algebraic structure
$(\wPl,\star,*)$ is a unital involutive complex algebra.
The unity of $\wPl$ is the empty diagram:
$$\includegraphics[width=1.5cm]{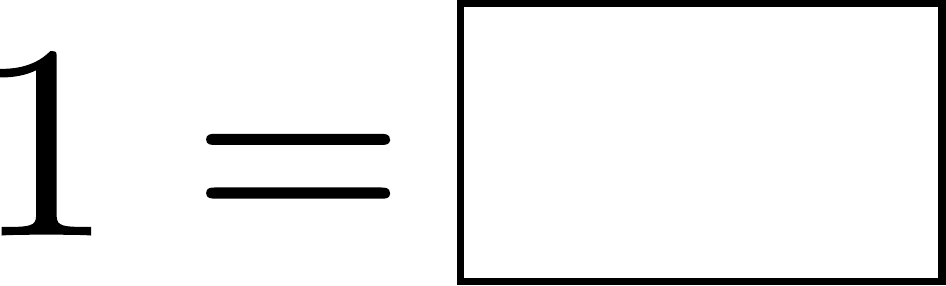}.$$
\end{prop}

\begin{proof}
The proof is the same as in the subfactor planar algebras case, see \cite{JSW_orthogonal_approach_planar_algebra}[proposition 3.2.].
\end{proof}

We define a trace on $Gr(\Pl)$ by the formula $tr(a)=\langle a,1\rangle$ so that the trace of an element is its zero-graded piece.
The inner product $\langle a,b\rangle$ is clearly equal to $tr(ab^*)$ and is positive definite by definition of an unshaded planar algebra.
Let $\h$ be the Hilbert space equal to the completion of $\wPl$ for the inner product $\langle\cdot ,\cdot\rangle$.
We denote its trace by $\Vert\cdot\Vert_\h$.
The Hilbert space $\h$ is equal to the orthogonal direct sum $\bigoplus_{n=0}^\infty\Pl_n$.
We prove in the next proposition that the left multiplication by elements of the graded vector space $\wPl$ is bounded for the preHilbert space structure.

\begin{prop}\label{prop_GJSW_mult_cont}
If $a\in \wPl$, there exists a positive constant $C>0$ such that, for any
$b\in \wPl,\,\Vert a\star b\Vert_\h\leqslant C\Vert b\Vert_\h$.
\end{prop}

\begin{proof}
Let $j\geqslant 0$, and consider the vector space $\Pl_{2j}$.
We equip $\Pl_{2j}$ with the product $\times$ defined as follows:
$$\includegraphics[width=3.5cm]{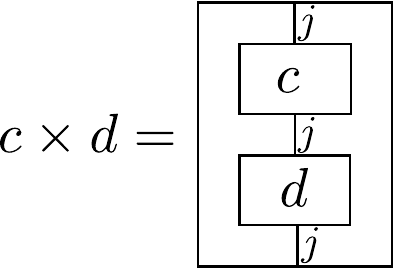},$$
where $c,d\in\Pl_{2j}$.
We equip $\Pl_{2j}$ with the involution $*$ coming from the planar algebra structure of $\Pl$.
The tangle acts \wrt\ the $*$-structure of $\Pl_{2j}$ by definition of an unshaded planar algebra.
Hence, for any $c,d\in\Pl_{2j}$ we have that $(c\times d)^*=d^*\times c^*$.
Let $\Vert\cdot\Vert_{\Pl_{2j}}$ be the norm:
$$\Vert a\Vert_{\Pl_{2j}}=\sup\{\Vert a\times d\Vert_\h,\,d\in\Pl_{2j},\,\Vert d\Vert_\h=1\},$$
where $a\in\Pl_{2j}$.
The $*$-algebra $(\Pl_{2j},\times,*)$ with the norm $\Vert\cdot\Vert_{\Pl_{2j}}$ is a (finite dimensional) $C^*$-algebra.

Let $n\geqslant 0$ and $j\leqslant n$, consider an element $a\in \Pl_n$ and
$$\includegraphics[width=3cm]{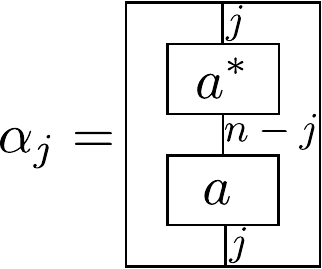}$$
in $\Pl_{2j}$.

Let us show that $\alpha_j$ is a positive element of the $C^*$-algebra $\Pl_{2j}$.
For this, we prove that for any $d\in \Pl_{2j}$,
 $\langle \alpha_j\times d,d\rangle\geqslant 0.$\\
We have that
$$\includegraphics[width=6cm]{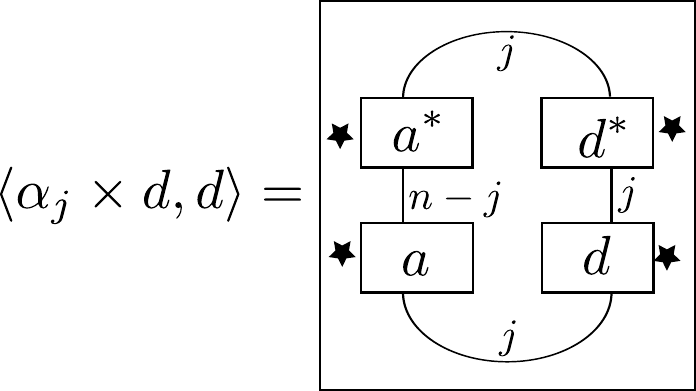}.$$
If we denote
$$\includegraphics[width=5cm]{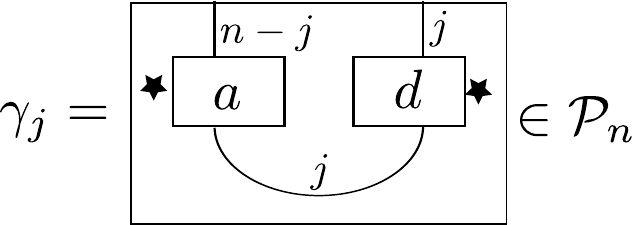},$$
we have that
$$\langle \alpha_j\times d,d\rangle=\Vert \gamma_j\Vert_\h^2\geqslant 0,$$
this tells us that $\alpha_j$ is positive.

Let us show that there exists a positive constant $C>0$, such that for any $m\geqslant 0$ and any vector $b\in\Pl_m$, we have that
$\Vert a\star b\Vert_\h\leqslant C\Vert b\Vert_\h.$
Let $j\leqslant\min(n,m)$, we have:
$$\includegraphics[width=8.5cm]{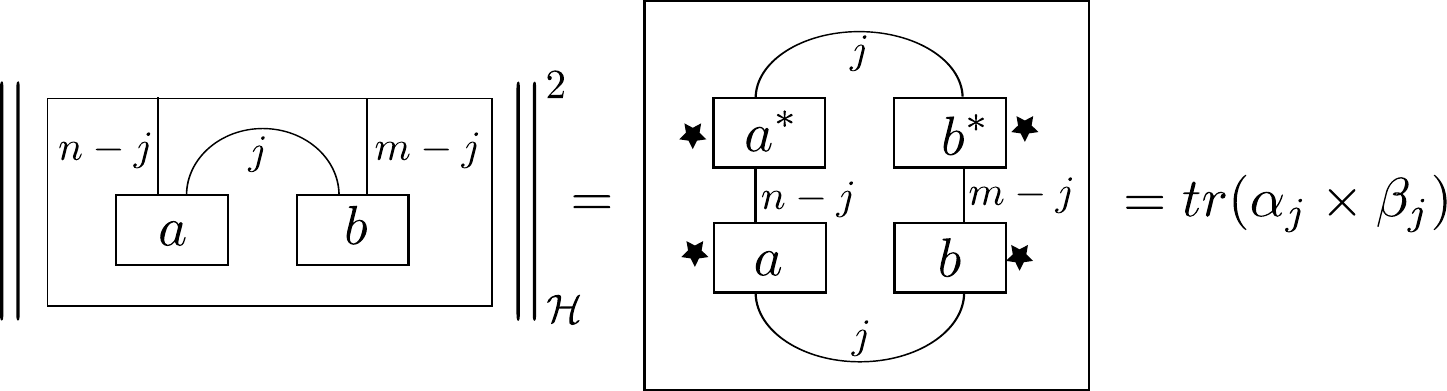}$$
where
$$\includegraphics[width=4cm]{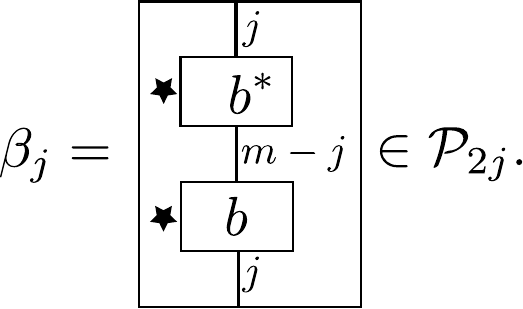}$$
We know that $\alpha_j$ and $\beta_j$ are positive operators of $\Pl_{2j}$; thus, there exists two self-adjoint elements $a_j,b_j\in\Pl_{2j}$ such that $a_j\times a_j=\alpha_j$ and $b_j\times b_j=\beta_j$.
If we look at the trace of the planar algebra $\Pl$:
\begin{align*}
tr(\alpha_j\times \beta_j)&=tr(a_j\times b_j\times b_j\times a_j)\\
&=\Vert a_j\times b_j\Vert_\h^2\leqslant \Vert a_j\Vert_{\Pl_{2j}}^2.\Vert b_j\Vert_\h^2.
\end{align*}
Clearly, for any $j$ we have that $\Vert b_j\Vert_\h=\Vert b\Vert_\h$.
Hence,
$$\Vert a\star b\Vert_\h\leqslant (\sum_{0\leqslant j\leqslant n}\Vert a_j\Vert_{\Pl_{2j}}).\Vert b\Vert_\h.$$
Thus for any $m\geqslant 0$ and any $b\in \Pl_m$,
$\Vert a\star b\Vert_\h\leqslant C\Vert b\Vert_\h,$
where
$C=\sum_{0\leqslant j\leqslant n}\Vert a_j\Vert_{\Pl_{2j}}$.
\end{proof}

We can define a representation of the $*$-algebra $\wPl$.
Consider the left multiplication $\pi:\wPl\longrightarrow \mathcal B(\h)$, such that $\pi(a)(b)=a\star b$ for any $a\in\wPl$ and $b\in\wPl$.
We write $M_\Pl$ the \vna\, generated by $\pi(\wPl)$.
We extend the representation to $M_\Pl$ and still denote it by $\pi$.
The right multiplication is also bounded, this gives us a representation of the opposite algebra:
$\rho:M_\Pl^{\text{op}}\longrightarrow \mathcal B(\h)$ such that $\rho(a)(b)=b\star a$.

\begin{rem}
The trace $tr$ of the graded algebra $Gr(\Pl)$ can be extend on the \vna\ $\MP$ with the formula $tr(a)=\langle a,1\rangle$.
It gives normal faithful trace on $\MP$ that we still denote by $tr$, hence $\MP$ is a finite \vna.

Consider the GNS representation of $\MP$ on the \Hs\ $\LMP$ associated to the trace $tr$.
This representation is conjugate with the representation $\pi:\MP\longrightarrow \Bo(\h)$.
We identify those two representations.
Furthermore, we identify the \vna\ $\MP$ and its dense image in the \Hs\ $\LMP$.
\end{rem}

We have the equality $\Vert a\bullet b\Vert_2=\Vert a\Vert_2\Vert b\Vert_2$ if $a\in\Pl_n$ and $b\in\Pl_m$.
By the triangle inequality, the bilinear function
\begin{align*}
Gr(\Pl)\times Gr(\Pl)&\longrightarrow Gr(\Pl)\\
(a,b)&\longmapsto a\bullet b
\end{align*}
is continuous for the norm $\Vert \cdot\Vert_2$.
We extend this operation on $L^2(M_\Pl)\times L^2(M_\Pl)$ and still denote it by $\bullet$.

%%%%%%%%%%%%%%%%%%%%%%%%%%%%%%%%%%%%%%%%%%%%%%%%%%%%%%%%%%%%%%%%%%%%%%%%%%%%%%%%%%%%%%%%%%%%%%%%%%%%%%%%%%%%%%%%%%%%%%%%%%%%%%%%%%%%%%%%%%%%%%%%%

\section{The cup subalgebra}\label{section_GJSW_cup}

%%%%%%%%%%%%%%%%%%%%%%%%%%%%%%%%%%%%%%%%%%%%%%%%%%%%%%%%%%%%%%%%%%%%%%%%%%%%%%%%%%%%%%%%%%%%%%%%%%%%%%%%%%%%%%%%%%%%%%%%%%%%%%%%%%%%%%%%%%%%%%%%%

Let $A\subset M_\Pl$ be the abelian \vna\, generated by the element cup $\cup\in\Pl_2$.
We call it the cup subalgebra.

%%%%%%%%%%%%%%%%%%%%%%%%%%%%%%%%%%%%%%%%%%%%%%%%%%%%%%%%%%%%%%%%%%%%%%%%%%%%%%%%%%%%%%%%%%%%%%%%%%%%%%%%%%%%%%%%%%%%%%%%%%%%%%%%%%%%%%%%%%%%%%%%%

\subsection{The bimodule structure of $_AL^2(M_\Pl)_A$}\label{subsection_GJSW_cup_coarse}

%%%%%%%%%%%%%%%%%%%%%%%%%%%%%%%%%%%%%%%%%%%%%%%%%%%%%%%%%%%%%%%%%%%%%%%%%%%%%%%%%%%%%%%%%%%%%%%%%%%%%%%%%%%%%%%%%%%%%%%%%%%%%%%%%%%%%%%%%%%%%%%%%

Let $n\geqslant 2$ and $V_n$ be the subspace of $\Pl_n$ of elements which vanish when a cap is placed at the top right and vanish when a cap is placed at the top left, i.e.
$$\includegraphics[width=8cm]{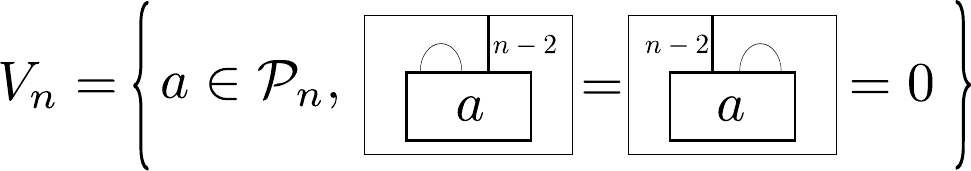}. $$
Let $$V=\bigoplus_{n=2}^\infty V_n,$$
be the orthogonal direct sum in $L^2(M_\Pl)$.
We consider the bimodule generated by $V$ that we denote by $\AVA$.

Let us write $_AL^2(M_\Pl)_A$ as a direct sum of bimodules:

\begin{prop}\label{prop_GJSW_LM_sum_V_P_A}
The bimodule $_AL^2(M_\Pl)_A$ is isomorphic to the direct sum
$$\ALAA\oplus\APA\oplus\AVA,$$
where $\APA$ is the bimodule generated by the $1$-box space $\Pl_1$.
\end{prop}

\begin{proof}
The subspace $\LA\subset \LMP$ is a bimodule and, by definition, $\APA$ and $\AVA$ are subbimodules of $\LMP$.
Hence to prove the proposition, it is sufficient to show that the Hilbert space $L^2(M_\Pl)$ is equal to the orthogonal direct sum
$$L^2(A)\oplus \APA\oplus\AVA.$$
Consider the three closed vector subspaces $E_1,E_2,E_3$ of $\LMP$ where
\begin{itemize}
  \item $E_1$ is spanned by the family of vectors $\{\cupbul k,\,k\geqslant 0\}$,\\
  \item $E_2$ is spanned by the family of vectors $\{\cupbul r\bullet b\bullet \cupbul k,\,b\in \Pl_1,\,r,k\geqslant 0\}$ and\\
  \item $E_3$ is spanned by the family of vectors $\{\cupbul r \bullet v\bullet \cupbul k,\,r,k\geqslant 0,\,v\in V\}$.
\end{itemize}
Let us show that $E_1=\LA$.
By definition of the space $E_1$, $\cup\in E_1$.
It is easy to see that $E_1$ is a bimodule.
Hence the bimodule generated by $\cup$, which is $\LA$, is included in $E_1$.
For the converse inclusion, an easy induction on $k$ shows that $\cupbul k\in A$.

Let us show that $E_2=\APA$.
By definition of the space $E_2$, $\Pl_1\subset E_1$.
It is easy to see that $E_2$ is a bimodule.
Hence the bimodule generated by $\Pl_1$, which is $\APA$, is included in $E_2$.
For the converse inclusion, an easy induction on $r$ and $k$ shows that $\cupbul r\bullet b\bullet\cupbul k\in \APA$ for any $b\in \Pl_1$.

Let us show that $E_3=\AVA$.
By definition, $V\subset E_3$ and clearly $E_3$ is stable by the actions of $\cup$, hence is a bimodule.
Therefore, $\AVA\subset E_3$.
For the converse inclusion, fix a $v\in V$.
An easy induction show that for any $r,k\geqslant 0$, $\cupbul r \bullet v\bullet \cupbul k\in\AVA$.

Consider the \Hs\ equal to the sum $E=E_1+E_2+E_3$.
Let us show that $E=L^2(M_\Pl)$.
To do this, we show that $\Pl_n\subset E$ for any $n\geqslant 0$.
We proceed by induction on $n$.

By definition, $\Pl_0\subset E_1$ and $\Pl_1\subset E_2$.
Consider $n\geqslant 2$ and suppose that $\Pl_{n-1}$ and $\Pl_{n-2}$ are included in $E$.
Let us denote the orthogonal of $V_n$ inside $\Pl_n$ by $W_n$.
By \cite{JSW_orthogonal_approach_planar_algebra}[Lemma 4.5.], we have that $W_n$ is the space spanned by element that can be written $y\bullet \cup$ and $\cup\bullet z$, where $y,z\in\Pl_{n-2}$.
It is easy to see that for any $x\in E$, we have that $\cup\bullet x\in E$ and $x\bullet \cup\in E$.
So $W_n\subset E$; thus, by definition of $E_3$, $V_n\subset E$.
So $\Pl_n\subset E$, we have proved that $E$ is a dense subspace of $L^2(M_\Pl)$.

Let us show that the $E_i$ are pairwise orthogonal.
Let $n\geqslant 2$, $v\in V_n$, $b\in \Pl_1$, and $k,l,r,m\geqslant 0$.
Consider the vectors $\cupbul r \bullet v\bullet \cupbul k$ and $\cupbul l \bullet b\bullet \cupbul m$.
They are in the vector spaces $\Pl_{2(r+k)+n}$ and $ \Pl_{2(l+m)+1}$.
The spaces $\{\Pl_n,\,n\geqslant 0\}$ are pairwise orthogonal by definition of the inner product on the graded vector space $Gr(\Pl)$.
So $\cupbul r \bullet v\bullet \cupbul k$ and $\cupbul l \bullet b\bullet \cupbul m$ are orthogonal if $2(r+k)+n\neq 2(l+m)+1$.
Suppose $2(r+k)+n= 2(l+m)+1$, by hypothesis $n\geqslant 2$, hence $l>r$ or $m>k$, in any case, $v$ will get a cap at the top right or left in the inner product $\langle \cupbul r \bullet v\bullet \cupbul k,\cupbul l \bullet b\bullet \cupbul m\rangle$.
Therefore, they are orthogonal and so are $E_2$ and $E_3$.

Let $r,k,l\geqslant 0$, $n\geqslant 2$ and $v\in V_n$.
Consider the inner product $\langle \cupbul r\bullet v\bullet \cupbul k,\cupbul l\rangle.$
It is equal to $0$ if $2r+2k+n\neq 2l$ by the orthogonality of the spaces $P_i$.
Suppose $2r+2k+n= 2l$, we have that $n$ is an even number and
$$\langle \cupbul r\bullet v\bullet \cupbul k,\cupbul l\rangle=\delta^{r+k}\langle v,\cupbul{\frac{n}{2}}\rangle.$$
This must be equal to $0$ because $v$ will get a cap at the top left in this inner product; thus, $E_1$ is orthogonal to $E_3$.

The space $E_1$ is included in the orthogonal direct sum $\bigoplus_n\Pl_{2n}$, and $E_2$ is included in the orthogonal direct sum $\bigoplus_n\Pl_{2n+1}$, therefore $E_1\perp E_2$.
Hence, the vector spaces $E_i$ are in direct sum and their sum is equal to the \Hs\ $\LMP$.
We have proved that $L^2(M_\Pl)$ is equal to the orthogonal direct sum
$$L^2(A)\oplus \APA\oplus\AVA.$$
\end{proof}

\begin{prop}\label{prop_GJSW_AVA}
The map
\begin{align*}
\eta_V : {\AVA} & \rightarrow  \lN\otimes V\otimes \lN \\
\dkr \cupbul k\bullet v\bullet \cupbul r & \mapsto  e_k\otimes v\otimes e_r
\end{align*}
defined a unitary transformation such that
\begin{align}\label{equa_GJSW_entrelacement_cup}
\eta_V\pi(\U)\eta_V^*&=(\sh)\otimes 1_V\otimes 1_\lN \text{ and}\\
\eta_V\rho(\U)\eta_V^*&=1_\lN\otimes 1_V\otimes (\sh).
\end{align}
In other words, the map $\U\longmapsto \sh$ defines a normal faithful representation of the \vna\ $A$ on the \Hs\ $\lN$ and the bimodule $\AVA$ is isomorphic to a direct sum of coarse correspondences.
\end{prop}

\begin{proof}
This is the same proof that the one given in \cite{JSW_orthogonal_approach_planar_algebra}[theorem 4.9.] for a subfactor planar algebra.
\end{proof}
This proposition is telling us that to understand the bimodule structure of $_AL^2(M_\Pl)_A$ is the same that to understand the bimodule structure of $\APA$.

\subsection{The bimodule generated by the $1$-box space}\label{subsection_GJSW_bimodule_struct_P_1}

We suppose, in this section, that $\Pl_1\neq\{0\}$.
Let us fix an element $b\in\Pl_1$ such that $\Vert b\Vert_2=1$.
We denote the bimodule generated by $b$ by $_A{\overline b}_A$.

\begin{lemm}\label{lem_ortho_basis_AbA}
The set
$$E_b=\{\dr b\bullet \cup ^{\bullet r},\,r\geqslant 0\}\cup \{\dkr \cup^{\bullet k}\bullet Z_b\bullet\cup^{\bullet r},\,k,r\geqslant 0\}$$
is an orthonormal basis of the Hilbert space $\AbA$, where
$$Z_b=\frac{\cup\bullet b-\dmi b\bullet \cup}{\sqrt{\delta-\dmi}}.$$
\end{lemm}

\begin{proof}
Let us prove that $E_b$ is a family of unit vectors.
Note that
$$\Vertvert{a\bullet b}_2=\Vertvert{a}_2\Vert b\Vert_2.$$
This implies that
$$\Vertvert{\deltamoin{r}b\bullet \cupbul{r}}_2=\deltamoin r \Vertvert b_2 \Vertvert\cup_2^r=\Vertvert b_2=1,$$
and
$$\Vertvert{\deltamoin{k+r}\cupbul{ k} \bullet Z_b\bullet \cupbul {r}}_2=\Vertvert{ Z_b}_2.$$
Let us compute $\Vertvert{ Z_b}_2$:
We have that
$$\Vertvert{ Z_b }_2^2=\frac{1}{\delta-\delta^{-1}}(\Vertvert{ \cup\bullet b}^2_2+\delta^{-2}\Vertvert{b\bullet \cup}_2^2-\delta^{-1}\langle \cup\bullet b,b\bullet \cup\rangle-\delta^{-1}\langle b\bullet \cup, \cup\bullet b\rangle).$$
The inner product
$$\includegraphics[width=7cm]{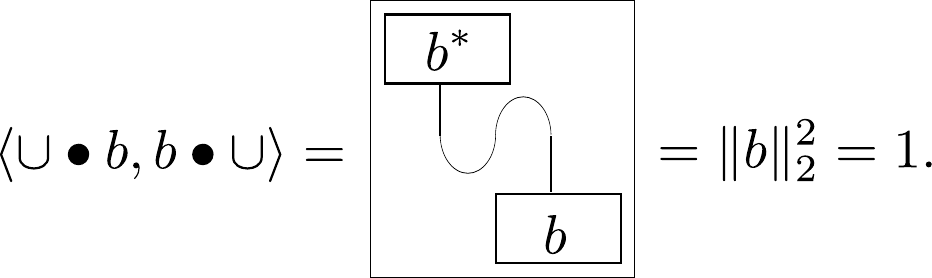}$$
Hence, $\langle b\bullet \cup, \cup\bullet b\rangle=\overline{\langle \cup\bullet b,b\bullet \cup\rangle}=1;$
thus, $\Vertvert{Z_b}_2=1$. We have proved that $E_b$ is a family of unit vectors.

Let us show that $E_b$ is an orthogonal family.
Consider the family of vectors
$$\{\delta^{-\frac{k+r}{2}}\cupbul k\bullet Z_b\bullet\cupbul r,\,k,r\geqslant 0\}.$$
Let $k,r,k_1,r_1,k_2,r_2\geqslant 0$ be some natural numbers.
The spaces $\Pl_n$ are pairwise orthogonal and $\cupbul k\bullet Z_b\bullet\cupbul r\in \Pl_{2(k+r)+3}$, thus
$$\langle \cupbul{r_1}\bullet Z_b\bullet\cupbul{k_1},\cupbul{r_2}\bullet Z_b\bullet\cupbul{k_2}\rangle=0 \,\,\text{if}\,\,  k_1+r_1\neq k_2+r_2.$$
Suppose $k_1+r_1= k_2+r_2$, and $r_1\neq r_2$,
$$\includegraphics[width=13cm]{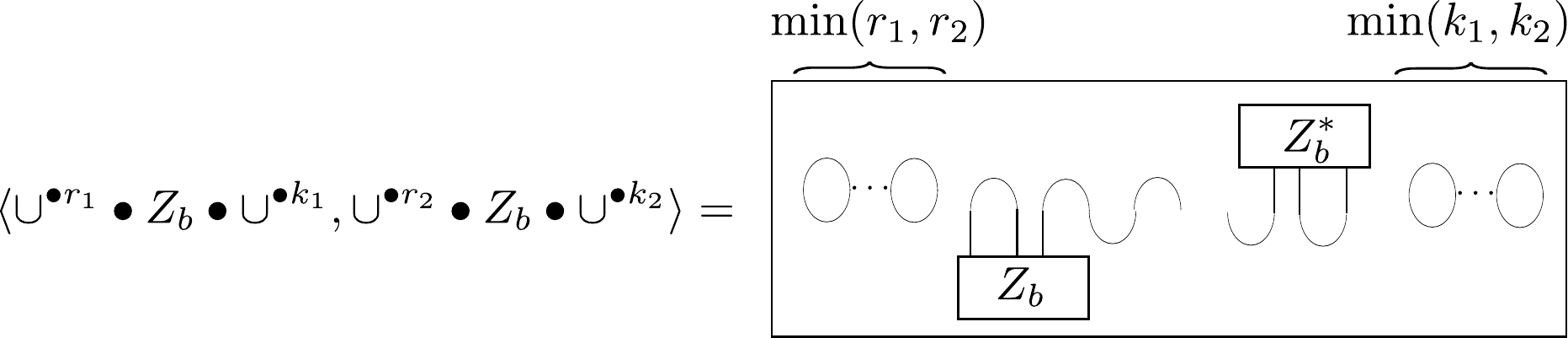}$$
We have that
$$\includegraphics[width=10cm]{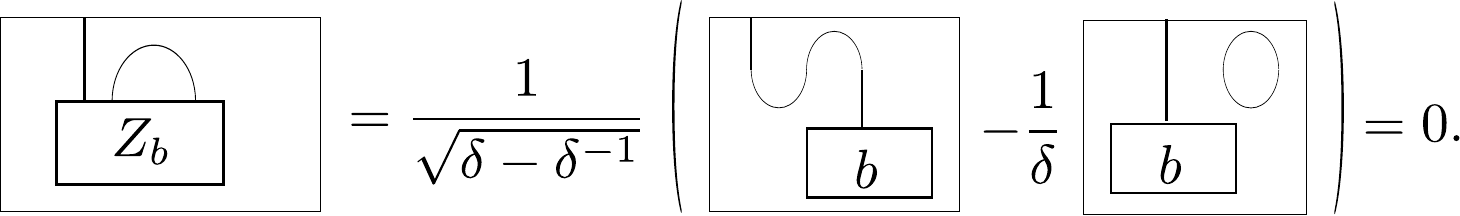}$$
Thus,
$$\langle \cupbul{r_1}\bullet Z_b\bullet\cupbul{k_1},\cupbul{r_2}\bullet Z_b\bullet\cupbul{k_2}\rangle=0.$$
Hence, the inner product
$\langle \cupbul{r_1}\bullet Z_b\bullet \cupbul{k_1},\cupbul{r_2}\bullet Z_b\bullet \cupbul{k_2}\rangle$ is non null
if and only if $r_1=r_2$ and $k_1=k_2$.
This is telling us that $\{\cupbul r\bullet Z_b\bullet \cupbul k,\,r,k\geqslant 0\}$ is an orthogonal family and then an orthonormal family of vectors.

Let us show that $\{b\bullet \cupbul r,\,r\geqslant 0\}$ is an orthogonal family.
The elements $b\bullet \cupbul{r_1}$ and $b\bullet \cupbul{r_2}$ belong to $\Pl_{1+2r_1}$ and $\Pl_{1+2r_2}$, so they are orthogonal if $r_1\neq r_2$.

Let us show that the two sets
$$\{\dr b\bullet \cup ^{\bullet r},\,r\geqslant 0\}$$
and
$$\{\dkr \cup^{\bullet k}\bullet Z_b\bullet\cup^{\bullet r},\,k,r\geqslant 0\}$$
are orthogonal.
Consider $ b\bullet \cupbul{r_1} $ and $\cupbul k\bullet Z_b\bullet\cupbul{r_2}$.
The vector $b\bullet \cupbul{r_1}$ belongs to $\Pl_{2r_1+1}$ and $\cupbul k\bullet Z_b\bullet\cupbul{r_2}\in\Pl_{2(k+r_2)+3}$.
These are orthogonal if $r_1\neq k+r_2+1$.
Supposing $r_1= k+r_2+1$, we get that $r_1\geqslant r_2+1.$
Thus,
$$\includegraphics[width=10cm]{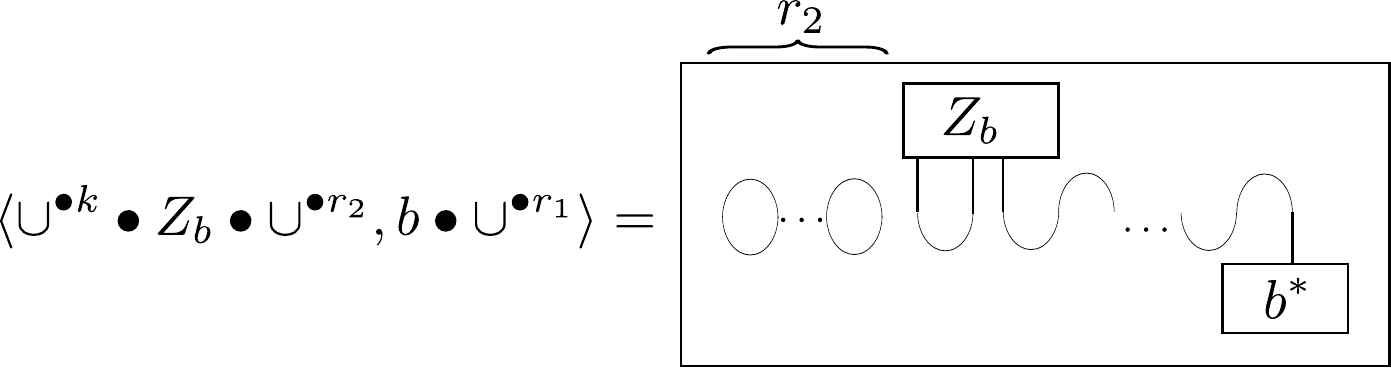}.$$
This is equal to zero because $Z_b$ has a cap on the top right.

We have shown that $E_b$ is an orthonormal family.

Let us show that the closed vector space spanned by $E_b$, that we denote by $X_b$, is equal to $\AbA$.
We see clearly that $X_b$ is stable by left and right multiplication by $\cup$, hence $X_b$ is a bimodule.
The space $X_b$ contains $b$, therefore it contains the bimodule generated by $b$ which is $\AbA$.\\
Let us show that $X_b$ is contained in $\AbA$, which is equivalent to showing that $E_b$ is included in $\AbA$.
The space $X_b$ is spanned by the family of vectors $\{\cupbul r \bullet b\bullet\cupbul k,\,r,k\geqslant 0\}$.
An easy induction on $r$ and $k$ shows that $\cupbul r \bullet b\bullet\cupbul k\in\AbA$ for any $r,k$.
\end{proof}

\begin{prop}\label{prop_decomposition_bimodule_b}
Consider the operator $\eta_b:\AbA\rightarrow \lNN$ defined as follows:
\begin{align*}
\eta_b(\dr b\bullet \cupbul r)&=e_0\otimes e_r \text{ and}\\
\eta_b(\dkr \cupbul k\bullet Z_b\bullet \cupbul r)&=e_{k+1}\otimes e_r.
\end{align*}
This operator $\eta_b$ is a unitary transformation.
Furthermore,
\begin{align}\label{equation_etab_entrelacement}
\eta_b \pi(\U) \eta_b^*&=\alpha+(s+s^*)\otimes 1 \text{ and} \\
\eta_b \rho(\U) \eta_b^*&=1\otimes (\sh),
\end{align}
where $\alpha$ is the operator of $\lNN$ defined as follows:
for all $x\in \lN,$
\begin{align*}
\alpha(e_0\otimes x)&=(\sqrt{1-\dmm} -1)(e_1\otimes x)+\dmi (e_0\otimes (\sh)(x))\\
\alpha(e_1\otimes x)&=(\sqrt{1-\dmm} -1)(e_0\otimes x)\\
\alpha(e_k\otimes x)&=0 \text{ if } k\geqslant 2.
\end{align*}
\end{prop}

\begin{proof}
We have proved that $E_b$ is an orthonormal basis of the Hilbert space $\AbA$; thus, the operator $\eta_b$ sends an orthonormal basis to an other one.
Hence $\eta_b$ is a unitary transformation.

Let us show that $\eta_b$ satisfies the equality \ref{equation_etab_entrelacement}:
Let us look at the left action of cup.
Let us compute $\cup\star (b\bullet \cupbul r)$:\\
If $r=0$,
$$\includegraphics[width=8cm]{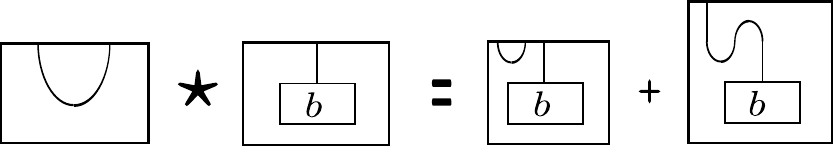}$$
If $r\geqslant 1$,
$$\includegraphics[width=10cm]{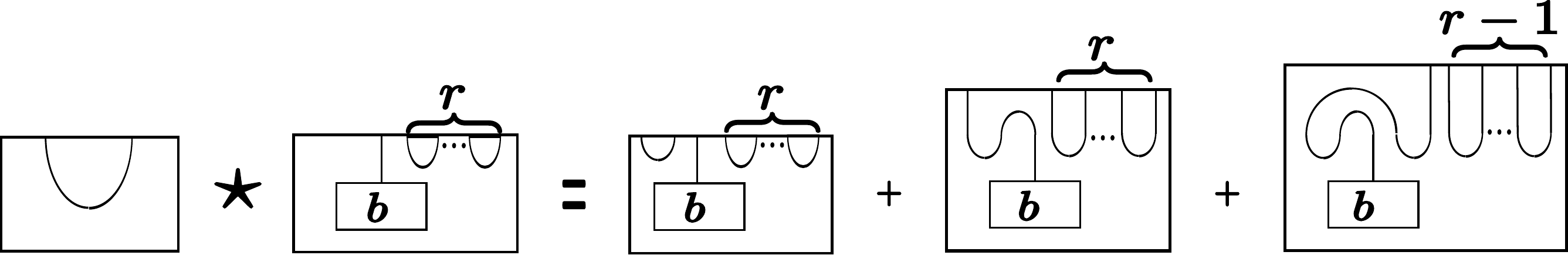}$$
Thus for any $r\geqslant 0$,
$$\cup\star (b\bullet \cupbul r)=\sqrt{\de-\dmi} (Z_b \bullet \cupbul r)+\dmi (b\bullet \cupbul{r+1})+(b\bullet \cupbul r)+(b\bullet \cupbul{r-1}).$$
So
$$(\U)\star (\dr b\bullet \cupbul r)=\sqrt{1-\dmm} (\dr Z_b\bullet \cupbul r)+\dmi (\de^{-\frac{r+1}{2}}b\bullet \cupbul{r+1})+\dmi (\de^{-\frac{r-1}{2}}b\bullet \cupbul{r-1}).$$
Let us compute $\cup\star (\cupbul k\bullet Z_b\bullet\cupbul r)$:\\
If $k=0$,
$$\includegraphics[width=12cm]{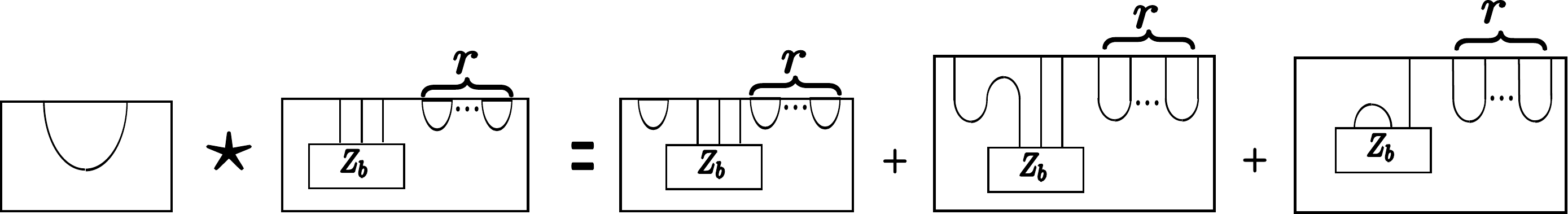}$$
and
$$\includegraphics[width=8cm]{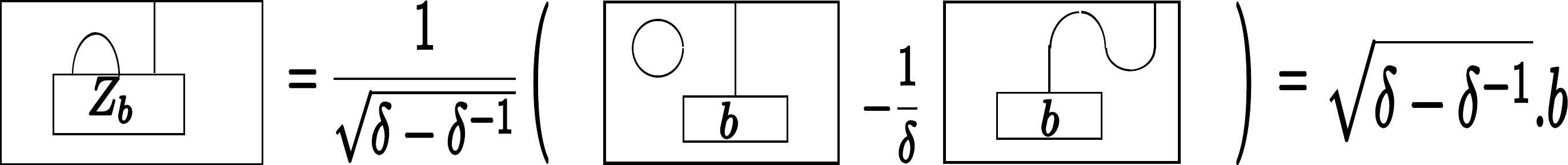}.$$
So,
$$(\U)\star (\dr Z_b\bullet \cupbul r)=\de^{-\frac{r+1}{2}}(\cup \bullet Z_b\bullet \cupbul r)+\sqrt{1-\dmm}(\dr b\bullet \cupbul r).$$
If $k\geqslant 1$,
$$\includegraphics[width=14cm]{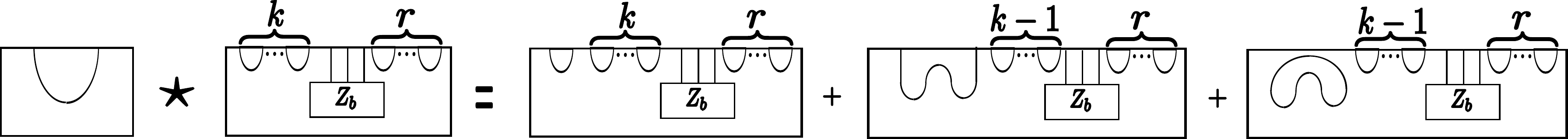}$$
Thus,
$$(\U)\star (\dkr \cupbul k\bullet Z_b\bullet \cupbul r)=(\de^{-\frac{k+1+r}{2}}\cupbul{k+1}\bullet Z_b\bullet \cupbul r)+(\de^{-\frac{k-1+r}{2}}\cupbul{k-1}\bullet Z_b\bullet \cupbul r),$$
where $k\geqslant 1$.

Let us look at the right action of cup:\\
Let us compute $(b\bullet \cupbul r)\star \cup$:\\
If $r=0$,
$$b\star \cup=(b\bullet \cup)+b.$$
If $r\geqslant 1$,
$$(b\bullet \cupbul r)\star \cup=(b\bullet \cupbul{r+1})+(b\bullet \cupbul r)+\de (b\bullet \cupbul{r-1}).$$
Let us compute $(\cupbul k\bullet Z_b\bullet \cupbul r)\star \cup$:\\
If $r=0$,
$$\includegraphics[width=12cm]{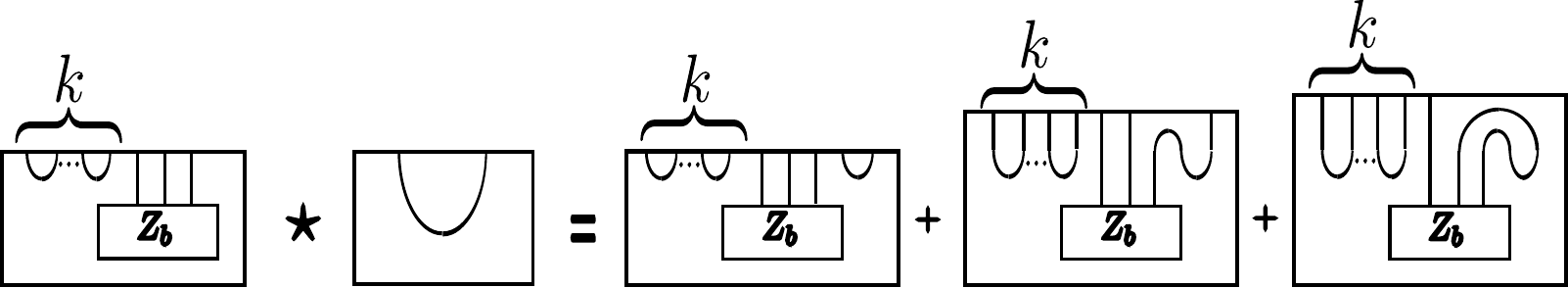}.$$
We have seen in the proof of lemma \ref{lem_ortho_basis_AbA} that $Z_b$ vanishes if it is capped off on the right.
Thus,
$$(\cupbul k \bullet Z_b)\star \cup=(\cupbul k\bullet Z_b\bullet\cup)+(\cupbul k\bullet Z_b).$$
If $r\geqslant 1$,
$$(\cupbul k\bullet Z_b\bullet \cupbul r)\star \cup=(\cupbul k\bullet Z_b\bullet \cupbul{r+1})+(\cupbul k\bullet Z_b\bullet\cupbul r)+\de (\cupbul k\bullet Z_b\bullet\cupbul{r-1})$$
Thus for any $k\geqslant 0,\,r\geqslant 0,$
\begin{align*}
(\dkr \cupbul k\bullet Z_b\bullet\cupbul r)\star (\U) & =\de^{-\frac{k+r+1}{2}}\cupbul{k}\bullet Z_b\bullet \cupbul{r+1}\\
& +\de^{-\frac{k+r-1}{2}}\cupbul{k}\bullet Z_b\bullet \cupbul{r-1}.
\end{align*}
\end{proof}

We want to show that the bimodules $\AbA$ is isomorphic to the coarse correspondence.
By proposition \ref{prop_decomposition_bimodule_b}, this is equivalent to saying that there exists a unitary $u$ of the Hilbert space $\lNN$ such that
\begin{align*}
u(\alpha+(\sh)\otimes 1)u^* & =(\sh)\otimes 1 \,\,\text{ and}\\
u(1\otimes (\sh))u^* & =1\otimes (\sh).
\end{align*}
So we are looking for a unitary $u$ that conjugates $\alpha+(\sh)\otimes 1$ and $(\sh)\otimes 1$, and that commutes with $1\otimes (\sh)$.
Let us consider the abelian von Neumann algebra
$$D=\{1\otimes (\sh)\}''\subset \mathcal B(\lNN)$$
generated by the operator $1\otimes (\sh)$.
It is equal to the \vna\ $D=\C.1\otimes D_0$, where $D_0=\{\sh\}''\subset \mathcal B(\lN)$.
Let us look at the distribution of the operator $\sh$.
\begin{prop}\label{prop_GJSW_semicircular_distribution_P_n}
The operator
$$\frac{\sh}{2}$$
is semicircular in the sense of Voiculescu \cite{Voiculescu_dykema_nica_Free_random_variables}.
In particular, the distribution of $\sh$ is absolutely continuous with respect to the Lebesgue measure, is supported in $\2$ and equal to:
$$d\nu(t)=\frac{\sqrt{4-t^2}}{2\pi}dt.$$
We have a unitary transformation $\eta_\nu:L^2([-2;2],\nu)\longrightarrow\lN$ defined on the dense subspace of continuous functions $\co$ by $f\in\co\longmapsto f(\sh)(e_0).$
This satisfies
\begin{equation}\label{equa_GJSW_etanu_sh}\eta_\nu (\sh)\eta_\nu^*(f)(t)=t f(t)\end{equation}
for any $f\in\co$ and $t\in\2$.
Furthermore, $\eta_\nu(P_n)=e_n,$ for all $n\geqslant 0$ where $\{P_n\}_n$ is the family of polynomials defined as follows:
$$\begin{array}{l}
P_0(X)=1\\
P_1(X)=X\\
P_n(X)=XP_{n-1}(X)-P_{n-2}(X),\,\,\text{for any}\,\, n\geqslant 2\\
\end{array},$$
where $X$ is an indeterminate.
\end{prop}

\begin{proof}
To show that the distribution of $\sh$ is $d\nu(t)=\frac{ \sqrt{4-t^2} }{ 2\pi }dt,$ see \cite{Voiculescu_dykema_nica_Free_random_variables}[Example 3.4.2].
The equality \ref{equa_GJSW_etanu_sh} is obvious by definition of $\eta_\nu$.

We prove that $\eta_\nu(P_n)=e_n$ by induction on $n$.

It is clear for $n=0$.\\
For $n=1$: $(\sh)(e_0)=e_1$, so it is true for $n=1$.\\
Let $n\geqslant 2$ and suppose the result true for $n-1$ and $n-2$.
We have:
\begin{align*}
P_n(\sh)(e_0)&=(\sh)P_{n-1}(\sh)(e_0)-P_{n-2}(\sh)(e_0)\\
&=(\sh)(e_{n-1})-e_{n-2}=e_n+e_{n-2}-e_{n-2}\\
&=e_n.
\end{align*}
By definition, $\eta_\nu$ is an isometry and we just proved that the orthonormal basis $\{e_n,\,n\geqslant 0\}$ is in the image of $\eta_\nu$.
Hence, $\eta_\nu$ is surjective; thus, it is a unitary transformation.
\end{proof}

The Hilbert space $\lN\otimes \Ln$ is the constant field of the Hilbert space $\lN$ over the probability space $(\2,\nu)$.
We identify it with the Hilbert space of measurable functions $\xi:\2\longrightarrow \lN$ that are square-integrable.
We denote such a vector of $\lN\otimes \Ln$ by the direct integral
$$\xi=\int_\2^\oplus \xi_t d\nu(t),$$
where $\xi_t\in\lN$.
A bounded measurable operator field $\{b_t,\,t\in \2\}$  defined a decomposable operator that we denote by
$$b=\int^\oplus_\2 b_td\nu(t).$$
It acts in the following way:
$$b(\xi)=\int^\oplus_\2 b_t(\xi_t)d\nu(t).$$
We recall that the vector space of decomposable operators is a \vna\ equal to the commutant of the diagonal algebra in $\Bo(\lN\otimes \Ln)$.

The unitary transformation
$$1\otimes \eta_\nu:\lN\otimes L^2(\2,\nu)\longrightarrow \lNN$$
conjugates the von Neumann algebra $D$ and the diagonal algebra associates to this decomposition.
The two operators $\pi(\U)$ and $\rho(\U)$ commute, this implies that the operator
$$c:=(1\otimes \eta_\nu)(\alpha+(\sh)\otimes 1)(1\otimes \eta_\nu)^*$$
commutes with the diagonal algebra, hence it is a decomposable operator.
We give in the next proposition an explicit decomposition of the operator $c$.

\begin{prop}\label{prop_GJSW_c_t}
The operator $c$ is equal to the direct integral
$$c=\int_\2^\oplus c_t d\nu(t),$$
 where $c_t$ acts on the Hilbert space $\lN$ and is equal to:
\begin{small}$$c_t=\left( \begin{array}{ccccccc}
\frac{t}{\delta}&\sqrt{1-\delta^{-2}}&0&0&0&0&...\\
\sqrt{1-\delta^{-2}}&0&1&0&0&0&...\\
0&1&0&1&0&0&...\\
0&0&1&0&1&0&...\\
0&0&0&1&0&1&...\\
...&...&...&...&...&...&...
\end{array}\right)$$\end{small} in the orthogonal standard basis of the \Hs\ $\lN$.
\end{prop}

\begin{proof}
Let us show that $\{c_t,\, t\in \2\}$ is a bounded measurable operator field.
For this we have to show that $t\mapsto c_t$ is measurable for the Borel $\sigma$-algebra generated by the operator strong topology on $\mathcal B(\lN)$ and that the set $\{\Vert c_t\Vert,\,t\in Y\}$ is bounded in $\mathbb R.$

Let $t\in \2$, the operator $c_t$ is a finite rank perturbation of $\sh$, so it is a bounded operator.
Consider the function from the compact interval $\2$ to the space of bounded linear operators $\mathcal B(\lN)$:
\begin{align*}
\2&\longrightarrow \mathcal B(\lN)\\
t&\longmapsto c_t.
\end{align*}
This function is clearly continuous for the norm topology on $\mathcal B(\lN)$.
This implies that its image is compact and then bounded in $\mathcal B(\lN)$.
This function is measurable for the Borel $\sigma$-algebra generated by the norm topology on $\mathcal B(\lN)$ and then is measurable for the $\sigma$-algebra generated by the operator strong topology.
This implies that the family $\{c_t,\,t\in\2\}$ is a bounded measurable field of operators.

Consider the decomposable operator
$$d=\int_\2^\oplus c_t d\nu(t)$$
acting on the Hilbert space $\lN\otimes \Ln$.
Let us show that the two operators $c$ and $d$ are equal, i.e.
$$c=\int_\2^\oplus c_t d\nu(t).$$
Let $\text{span}\{e_n,\,n\geqslant 0\}$ be the subspace of $\lN$ spanned by the standard basis and let $\co\subset L^2(\2,\nu)$ be the subspace of continuous functions.
The algebraic tensor product of those vectors spaces,
$$\text{span}\{e_n,\,n\geqslant 0\}\otimes\co,$$
is a dense subspace of $\lN\otimes \Ln$.
Let us show that $c$ and $d$ coincide on this subspace.

Consider the operator
$$c-(\sh)\otimes 1=(1\otimes \eta_\nu)\circ\alpha\circ(1\otimes \eta_\nu^*)\in\Bo(\lN\otimes\Ln)$$
and a vector $f\in \co$.
Let us compute the vector $(c-(\sh)\otimes 1)(e_k\otimes f):$\\
For $k=0$,
\begin{align*}
(c-(\sh)\otimes 1)(e_0\otimes f)&=(1\otimes \eta_\nu)(\alpha(e_0\otimes \eta_\nu^*(f))) \\
&=(1\otimes \eta_\nu)[(\sqrt{1-\delta^{-2}}-1)(e_1\otimes\eta_\nu^*(f))\\
&+\dmi e_0\otimes ((\sh)(\eta_\nu^*(f)))]\\
&=(\sqrt{1-\delta^{-2}}-1)(e_1\otimes f)+\dmi (e_0\otimes (\eta_\nu(\sh)\eta_\nu^*)(f)).
\end{align*}
For $k=1$,
$$(c-(\sh)\otimes 1)(e_1\otimes f)=(\sqrt{1-\delta^{-2}}-1)(e_0\otimes f).$$
For any $k\geqslant 2$,
$$(c-(\sh)\otimes 1)(e_k\otimes f)=0.$$
On the other hand let us compute
$$(d-(\sh)\otimes 1) (e_k\otimes f):$$
For $k=0$,
\begin{align*}
(d-(\sh)\otimes 1) (e_0\otimes f)&=\left(\int_\2^\oplus (c_t-s-s^*) d\nu(t)\right)\left(\int_\2^\oplus f(t)e_0 d\nu(t)\right)\\
&=\int_\2^\oplus f(t)(c_t-s-s^*)(e_0) d\nu(t)\\
&=\int_\2^\oplus( f(t)\frac{t}{\de} e_0+f(t)(\sqrt{1-\dmm} -1)e_1) d\nu(t)\\
&=\dmi (e_0\otimes \eta_\nu (\sh)\eta_\nu^*(f))+(\sqrt{1-\dmm} -1)(e_1\otimes f)\\
&=(c-(\sh)\otimes 1)(e_0\otimes f).
\end{align*}
For $k=1$,
\begin{align*}
(d-(\sh)\otimes 1)(e_1\otimes f)&=\int_\2^\oplus f(t)(c_t-s-s^*)(e_1) d\nu(t)\\
&=\int_\2^\oplus( f(t)(\sqrt{1-\dmm} -1)e_0) d\nu(t)\\
&=(\sqrt{1-\dmm}-1)(e_0\otimes f)\\
&=(c-(\sh)\otimes 1)(e_1\otimes f).
\end{align*}
Let $k\geqslant 2,$
\begin{align*}
(d-(\sh)\otimes 1)(e_k\otimes f)&=\int_\2^\oplus f(t)(c_t-s-s^*)(e_k) d\nu(t)\\
&=\int_\2^\oplus(0) d\nu(t)=0\\
&=(c-(\sh)\otimes 1)(e_k\otimes f).
\end{align*}
The two operators $$c \,\,\,\,\text{and}\,\,\,\, \int_\2^\oplus c_t d\nu(t)$$ coincide on the dense subspace
$$\text{span}\{e_n,\,n\geqslant 0\})\otimes\co$$
of $\lN\otimes L^2(\2,\nu)$; thus, they are equal.
\end{proof}

Let us show that $c_t$ is unitarily equivalent to $\sh$:
Before proving it, let us recall some definitions and basic facts on spectral theory.
\begin{defi}
Let $\h$ be a \Hs\ with a inner product $\langle\cdot,\cdot\rangle$ and $a\in\Bo(\h)$ a self-adjoint operator acting on $\h$.
We denote the spectrum of $a$ by $\sigma(a)$.
The essential spectrum of $a$ is the complement of the set of isolated eigenvalues of finite multiplicity in $\sigma(a)$.
We denote it by $\sigma_{ess}(a)$.

Suppose that $a$ is a self-adjoint operator.
For any vector $\xi\in\h$, one can associate a Radon measure $\mu_\xi$ on the spectrum of $a$.
Consider 
$\mu_\xi(f)=\langle f(a)\xi,\xi\rangle$ for any continuous functions on $\sigma(a)$, $f\in\mathcal C(\sigma(a))$.
Let $\h_{ac}$ be the \Hs\ of vectors $\xi\in\h$ such that the measure $\mu_\xi$ is absolutely continuous \wrt\ the Lebesgue measure.
Let $\h_{sc}$ be the \Hs\ of vectors $\xi\in\h$ such that the measure $\mu_\xi$ is singular \wrt\ the Lebesgue measure.
Let $\h_{pp}$ be the \Hs\ of vectors $\xi\in\h$ such that the measure $\mu_\xi$ is a pure point measure.
We have that $\h=\h_{ac}+\h_{sc}+\h_{pp}$ but the sum is not direct in general.
We define the absolutely continuous spectrum of $a$ by the spectrum of the operator $a$ restricted to the \Hs\ $\h_{ac}$.
We denote it by $\sigma_{ac}(a)$.
We define the singular spectrum and the pure point spectrum in the similar way.
We denote them by $\sigma_{sc}(a)$ and by $\sigma_{pp}(a)$.

The operator $a\in\Bh$ is said to have uniform multiplicity equal to $1$ if there exists a measure $\mu$ on the spectrum of $a$ and a unitary transformation
$w:\h\longrightarrow L^2(\sigma(a),\mu)$
such that for any function $f\in L^2(\sigma(a),\mu)$, $waw^*(f)(z)=zf(z)$, where $z\in\sigma(a)$.
In that case, the spectrums $\sigma_{ac}(a)$, $\sigma_{sc}(a)$ and $\sigma_{pp}(a)$ form a partition of the spectrum of $a$.
See \cite{Kato_perturbation_theory_lin_ope} for more details.
\end{defi}

\begin{lemm}\label{lem_c_g_equivalent_a_sh}
Let $t\in\2$, then $c_t$ is unitarily equivalent to $\sh$.
\end{lemm}

\begin{proof}
Let us fix $t\in\2$, and consider the operator $c_t$.
Let us show that the spectrum of the operator $c_t$, $\sigma(c_t)$, is equal to $\2$:

The operator $k_t:=c_t-(\sh)$ is of finite rank.
The theorem of Weyl-von Neumann, see \cite{Kato_perturbation_theory_lin_ope}[p.523], shows that the essential spectrum of an operator acting on a Hilbert space is invariant under compact perturbation.
So, the essential spectrum of $c_t$, $\sigma_{\text{ess}}(c_t)$, is equal to the essential spectrum of $\sh$, $\sigma_{\text{ess}}(\sh)$.
The operator $\sh$ is semicircular, its spectrum is essential and equal to $\2$; thus, $\sigma_{\text{ess}}(c_t)=\2$.
The complement of the essential spectrum inside the spectrum is equal to the set of isolated eigenvalues of finite multiplicity.
It is called the discrete spectrum.
Let us show that this complement is empty.
The operator $c_t$ is self-adjoint, so its spectrum is included in the real line $\mathbb R$.
If we show that $c_t$ does not have any real eigenvalue with module strictly bigger than 2, we will have shown that the discrete spectrum is empty.
Consider a real number $z>2$ and $x=(x_n)_n\in\lN$ such that $c_t(x)=zx$.
For any $n\geqslant 1$, $x_n+x_{n+2}=zx_{n+1}$ the roots of the characteristic polynomials of this equation are
$$r=\frac{z-\sqrt{z^2-4}}{2}$$
and
$$l=\frac{z+\sqrt{z^2-4}}{2}.$$
Hence, there exists two complex numbers $B,C\in\C$, such that $x_n=Br^n+Cl^n$ for any $n\geqslant 1$.
We notice that $\vert l\vert >1$, and $x$ is a square summable complex sequence; thus, $C=0$.
Hence, for any $n\geqslant 1$, $x_n=x_1r^{n-1}$.
The equality $c_t(x)=zx$ gives us the system of equations:
$$\begin{array}{c}
 \frac{t}{\de}x_0+\sqrt{1-\dmm} x_1=zx_0 \\
\sqrt{1-\dmm} x_0+x_2=zx_1
\end{array}.$$
This implies that
\begin{equation*}
    [(1-\dmm)-(\frac{z-\sqrt{z^2-4}}{2}-z)(\frac{t}{\de}-z)]x_1=0,
\end{equation*}
which means that $[(1-\dmm)-h(z)]x_1=0$ where
$$h(z)=(z-\frac{t}{\de})\frac{z+\sqrt{z^2-4}}{2}.$$
Since the function $h$ is strictly increasing on $\R_+$ and $h(2)\geqslant 2(1-\dmi)$, we have
$$h(z)-(1-\dmm)>(\dmi-1)^2>0.$$
Hence $x_1=0$, that implies that for any $n\geqslant 1$, $x_n=0$ and also that $x_0=0$.
For the case where $z<-2$ it is the same proof where we replace $z$ by $-z$.
Hence, the spectrum of $c_t$, $\sigma(c_t)=\sigma_{\text{ess}}(c_t)=\2$.\\

Let us show that $c_t$ is of uniform multiplicity equal to $1$.
To do this, we show that the vector $e_0$ of the Hilbert space $\lN$ is cyclic for the abelian von Neumann algebra generated by the self-adjoint operator $c_t$ in $\mathcal B(\lN)$.
Consider the family of polynomials $\{S_{n,t},\,n\geqslant 0\}$ defined as follows:
\begin{align*}
S_{0,t}(X)&=1\\
S_{1,t}(X)&=\dfrac{X-\frac{t}{\delta}}{\sqrt{1-\delta^{-2}}}\\
S_{2,t}(X)&=XS_{1,t}(X)-\sqrt{1-\delta^{-2}}\\
S_{n,t}(X)&=XS_{n-1,t}(X)-S_{n-2,t}(X),\,\,\text{for any} \,\,n\geqslant 3,
\end{align*}
where $X$ is an indeterminate.\\
Let us show that for any $n\geqslant 0$, $S_{n,t}(c_t)(e_0)=e_n$.
We proceed by induction:\\
For $n=0$ it is trivial.\\
For $n=1$:
\begin{align*}
S_{1,t}(c_t)(e_0)&=\frac{1}{\sqrt{1-\delta^{-2}}} (c_t-\frac{t}{\delta} .1)(e_0)\\
&=\frac{1}{\sqrt{1-\delta^{-2}}} (\frac{t}{\delta} e_0+\sqrt{1-\delta^{-2}} e_1-\frac{t}{\delta} e_0)\\
&=e_1.
\end{align*}
For $n=2$:
\begin{align*}
S_{2,t}(c_t)(e_0)&=c_t S_{1,t}(c_t)(e_0)-\sqrt{1-\delta^{-2}} e_0=c_t(e_1)-\sqrt{1-\delta^{-2}} e_0\\
&=\sqrt{1-\delta^{-2}} e_0+e_2-\sqrt{1-\delta^{-2}} e_0\\
&=e_2.
\end{align*}
Let $n\geqslant 3$, we suppose the result is true for $n-1$ and $n-2$:
\begin{align*}
S_{n,t}(c_t)(e_0)&=c_t S_{n-1,t}(c_t)(e_0)-S_{n-2,t}(c_t)(e_0)\\
&=c_t(e_{n-1})-e_{n-2}.
\end{align*}
We know that for any $k\geqslant 2$, $c_t(e_k)=e_{k-1}+e_{k+1}$, so  $S_{n,t}(c_t)(e_0)=e_n$.\\
This show that the vector $e_0$ is a cyclic vector for the von Neumann algebra generated by $c_t$.

Consider the Radon measure $\nu^t$ defined as follows:
$$\int_\R f(z)d\nu^t(z)=\langle f(c_t)(e_0),e_0\rangle,$$
for any continuous complex valued function $f:\R\rightarrow \C$.
The measure $\nu^t$ is the distribution of the operator $c_t$ and is supported on the spectrum of $c_t$ that is $\2$.
Consider the Hilbert space of square integrable functions $L^2(\2,\nu^t)$, we have an operator:
\begin{align*}
\eta_t:L^2(\2,\nu^t)&\longrightarrow \lN \\
f\in \co & \longmapsto f(c_t)(e_0).
\end{align*}
This operator $\eta_t$ is a unitary transformation by construction.
It satisfies that
$$(\eta_t c_t\eta_t^*)(f)(z)=zf(z),$$
for any continuous function $f\in \co$ and for $\nu$-\alev\ $z\in\2$.
Hence, $c_t$ is of uniform multiplicity equal to $1$.

This implies that the spectrum $\sigma(c_t)$ is equal to the \textit{disjoint} union of its pure point spectrum $\sigma_{\text{pp}}(c_t)$, its singular spectrum $\sigma_{\text{sc}}(c_t)$, and its absolutely continuous spectrum $\sigma_{\text{ac}}(c_t)$.

Let us show that the spectrum of $c_t$ is absolutely continuous.
The theorem of Kato-Rosenblum \cite{Kato_perturbation_theory_lin_ope}[p. 540] says the following:\\
If $a$ and $b$ are self-adjoint operators acting on a \Hs\ $\h$, and $a$ is trace class, then the absolutely continuous part of $b$ and $a+b$ are unitarily equivalent.\\
The operator $k_t=c_t-(\sh)$ is of finite rank, so is trace class.
Hence, the absolutely continuous part of $c_t$ and $\sh$ are unitarily equivalent.
In particular, the absolutely continuous spectrum of $c_t$ and $\sh$ are equal.
The operator $\sh$ is semicircular, so is equal to its absolutely continuous part.
The spectrum of $\sh$ is equal to $\2$, so the absolutely continuous spectrum of $c_t$ is equal to $\2$.
This implies that $\sigma(c_t)=\sigma_{\text{ac}}(c_t)=\2$ and then $\sigma_{\text{pp}}(c_t)=\sigma_{\text{sc}}(c_t)=\emptyset$.
Hence, $c_t$ is equal to its absolutely continuous part.

Let us apply again the theorem of Kato-Rosenblum, we get that $c_t$ and $\sh$ are unitarily equivalent.
\end{proof}

\begin{prop}\label{prop_GJSW_ope_unit_equiv}
There exists a unitary $u$ acting on the \Hs\ $\lNN$ that commutes with $1\otimes (\sh)$ and conjugates $\alpha+(\sh)\otimes 1$ with $(\sh)\otimes 1.$
Hence, the bimodule $\AbA$ is isomorphic to the coarse correspondence.
\end{prop}

\begin{proof}
By lemma \ref{lem_c_g_equivalent_a_sh}, for any $t\in \2$, the operator $c_t$ is unitarily equivalent to the operator $\sh$.
By \cite{Dixmier_livre_hilbert}[Lemma 2,Chap.II,\S 2] there exists a measurable operator field of unitaries $\{u_t,\,t\in\2\}$ such that $u_t c_t u_t^*=\sh$.
This defines a unitary operator
$$u=\int^\oplus _\2 u_t d\nu(t),$$
such that $u$ commutes with $1\otimes (\sh)$ because it is decomposable.
By construction $u(\alpha+(\sh)\otimes 1)u^*=(\sh)\otimes 1$.

Let us show that the bimodule $\AbA$ is isomorphic to the coarse correspondence.
Consider the unitary transformation $\eta_b:\AbA\longrightarrow \lNN$ given in proposition \ref{prop_decomposition_bimodule_b} and the unitary $u\in\Bo(\lNN)$ that we just consider.
Let $w=u\circ\eta_b$, it is a unitary transformation from the \Hs\ $\AbA$ into the \Hs\ $\lNN$.
Furthermore, by construction, it satisfies that
$$w\pi(\U)w^*=(\sh)\otimes 1$$
and
$$w\rho(\U)w^*=1\otimes(\sh).$$
Hence, by definition, $\AbA$ is isomorphic to the coarse correspondence via the unitary transformation $w$.
\end{proof}

Consider $_A\overline{\Pl_1}_A$ the subbimodule of $_AL^2(M_\Pl)_A$ generated by $\Pl_1$.

\begin{cor}\label{cor_APA}
The bimodules $_A\overline{\Pl_1}_A$ is isomorphic to a direct sum of coarse correspondences.
\end{cor}

\begin{proof}
Let $\{b^i,\,i\in I\}$ be an orthonormal basis of $\Pl_1$, view as a subspace of the Hilbert space $L^2(M_\Pl)$.
The bimodule $\APA$ is isomorphic to the direct sum of bimodules:
$$\bigoplus_{i\in I}{_A\overline{b^i}_A},$$
where $_A\overline{b^i}_A$ is the subbimodule of $\LMP$ generated by the vector $b^i$.
By proposition \ref{prop_GJSW_ope_unit_equiv} we have that the bimodules ${_A\overline{b^i}_A}$ is isomorphic to the coarse correspondence.
Hence, $\APA$ is isomorphic to a direct sum of coarse correspondences.
\end{proof}

\begin{theo}\label{theo_GJSW_main_theo}
The bimodule $_AL^2(M_\Pl)_A$ is isomorphic to the direct sum of the bimodule $_A\LA_A$ and some copies of the coarse correspondence.
In particular, $\MP$ is a II$_1$ factor and $\AM_\Pl$ is a MASA.
\end{theo}

\begin{proof}
By proposition \ref{prop_GJSW_AVA}, the bimodule $\AVA$ is isomorphic to a direct sum of coarse correspondences and by proposition \ref{prop_GJSW_ope_unit_equiv}, the bimodule $\APA$ is isomorphic to a direct sum of coarse correspondences.
By proposition \ref{prop_GJSW_LM_sum_V_P_A}, the bimodule $_A\LMP_A$ is isomorphic to the direct sum
$_A\LA_A\oplus\APA\oplus\AVA.$
Therefore, the bimodule $_A\LMP_A$ is isomorphic to the direct sum of the bimodule $_A\LA_A$ and a direct sum of coarse correspondences.

Let us show that $\AM_\Pl$ is a MASA.
For this, it is sufficient to show that for any vector $\xi\in\LMP$ such that $\pi(\cup)\xi=\rho(\cup)\xi$ we have that $\xi\in\LA$.
Consider a vector in the orthogonal of $\LA$, $\xi\in\LMP\ominus\LA$, and suppose that $\pi(\cup)\xi=\rho(\cup)\xi$.
We have seen that the $A$-bimodule $\LMP\ominus\LA$ is isomorphic to a direct sum of coarse correspondences.
Hence, we can consider that $\xi$ is a vector of the coarse correspondence $\lNN$ such that
\begin{equation}\label{equa_GJSW_kerderivsh}
    ((\sh)\otimes 1)(\xi)=(1\otimes(\sh))(\xi).
\end{equation}

Let $HS(\lN)$ be the \Hs\ of Hilbert-Schmidt operators on the \Hs\ $\lN$ and let $\Psi:\lNN\longrightarrow HS(\lN)$ be the anti-linear isomorphism
defined such that $$\Psi(x_1\otimes x_2)(x)=\langle x,x_1\rangle x_2,$$
where $x\in\lN$ and $x_1\otimes x_2\in\lNN$.
If we apply the transformation $\Psi$ to the equation \ref{equa_GJSW_kerderivsh}, we get that the two operators $\Psi(\xi)$ and $\sh$ commute.
Suppose that $\xi$ is a non null vector, then $\Psi(\xi)$ is a non null Hilbert-Schmidt operator.
The Hilbert-Schmidt operators are compact operators. Hence $\Psi(\xi)$ acts by homothety  on a non null finite dimensional vector space.
Hence, the self-adjoint operator $\sh$ leaves invariant a non null finite dimensional vector space.
This implies that $\sh$ admits an eigenvalue, but the spectrum of $\sh$ is absolutely continuous, a contradiction.
Therefore, $\xi=0$ and so $\AM_\Pl$ is a MASA.

Let us show that the \vna\ $\MP$ is a factor.
Let $a\in\MP$ be an element in the center of $\MP$, in particular, $a$ commutes with $A$.
Hence, $a\in A$ because $\AM_\Pl$ is a MASA.
For any vector $\xi\in\LMP$, we have that $\pi(a)\xi=\rho(a)\xi$ because $a$ commutes with $\MP$.
Consider the map $\U\longmapsto \sh$, it defines a faithful representation from the \vna\ $A$ on the \Hs\ $\lN$, see proposition \ref{prop_GJSW_AVA}.
Let $b\in\Bo(\lN)$ be the image of $a$ by this representation.
The orthogonal of $\LA$ is equal to a direct sum of coarse correspondences.
Then, for any $x\otimes y\in\lNN$,
$$b(x)\otimes y=x\otimes b(y).$$
Therefore,
\begin{align*}
\Vert b(x)\otimes y\Vert=\Vert b(x)\Vert\Vert y\Vert=\Vert x\otimes b(y)\Vert=\Vert x\Vert \Vert b(y)\Vert.
\end{align*}
Hence,
$$\Vert b(x)\otimes y\Vert^2=\Vert b(x)\Vert\Vert b(y)\Vert\Vert x\Vert \Vert y\Vert.$$
On the other hand,
\begin{align*}
\Vert b(x)\otimes y\Vert^2&=\langle b(x)\otimes y,x\otimes b(y)\rangle=\langle b(x),x\rangle\langle y,b(y)\rangle\\
&\leqslant \Vert b(x)\Vert\Vert b(y)\Vert\Vert x\Vert \Vert y\Vert,
\end{align*}
by the inequality of Cauchy-Schwartz.
We notice that we are in the case of equality of Cauchy-Schwartz, hence the vectors $b(x)$ and $x$ are homothetic.
Therefore, $b$ is an homothety, hence $a$ is an homothety.
\end{proof}

%%%%%%%%%%%%%%%%%%%%%%%%%%%%%%%%%%%%%%%%%%%%%%%%%%%%%%%%%%%%%%%%%%%%%%%%%%%%%%%%%%%%%%%%%%%%%%%%%%%%%%%%%%%%%%%%%%%%%%%%%%%%%%%%%%%%%%%%%%%%%%%%%

\section{Appendix}

%%%%%%%%%%%%%%%%%%%%%%%%%%%%%%%%%%%%%%%%%%%%%%%%%%%%%%%%%%%%%%%%%%%%%%%%%%%%%%%%%%%%%%%%%%%%%%%%%%%%%%%%%%%%%%%%%%%%%%%%%%%%%%%%%%%%%%%%%%%%%%%%%

\subsection{Subfactor planar algebra associated to an unshaded planar algebra}

%%%%%%%%%%%%%%%%%%%%%%%%%%%%%%%%%%%%%%%%%%%%%%%%%%%%%%%%%%%%%%%%%%%%%%%%%%%%%%%%%%%%%%%%%%%%%%%%%%%%%%%%%%%%%%%%%%%%%%%%%%%%%%%%%%%%%%%%%%%%%%%%%
By mimicking \cite{JSW_orthogonal_approach_planar_algebra} and using Theorem \ref{theo_GJSW_main_theo} we construct a tower of II$_1$ factors associated to a given \unplal\ $\Pl$.

Consider the \vesp\ 
$$\includegraphics[width=8cm]{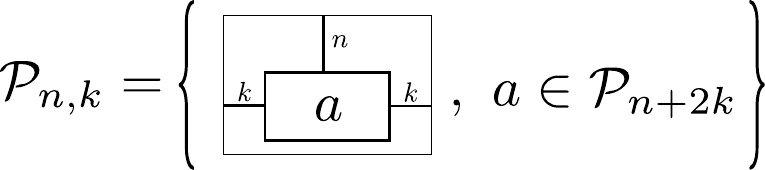},$$
where $n,k\geqslant 0$. 
As before, we assume that the stars are placed at the top left of the boxes.
We equip this vector space with the inner product of $\Pl_{n+2k}$.
Let $\Grk$ be the graded vector space equal to the direct sum of the $\Pl_{n,k}$. 
We extend the inner product of the $\Pl_{n,k}$ on $\Grk$ such that the $\Pl_{n,k}$ are pairwise orthogonal.
We denote by $\h_k$ the \Hs\ equal to the completion of $\Grk$ for the inner product.
Let $\wedge_k$ be the multiplication of $\Grk$ defined as follows:
$$\includegraphics[width=8cm]{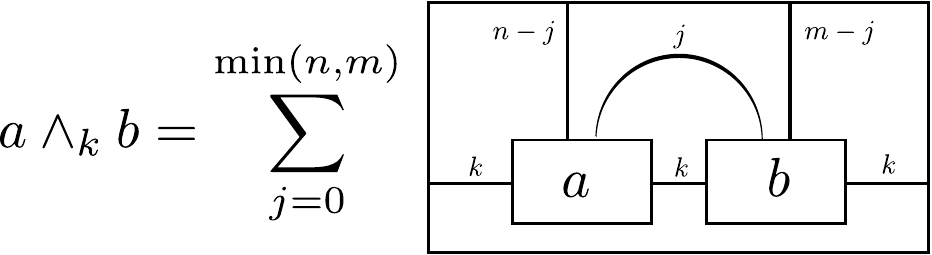},$$
where $a\in \Pl_{n,k}$ and $b\in\Pl_{m,k}$.
We equip $\Grk$ with the involution coming from the planar algebra structure.
Following \cite{JSW_orthogonal_approach_planar_algebra} and Proposition \ref{prop_GJSW_mult_cont} we have that $\wedge_k$ defined a bounded representation $\pi_k$ of the involutive algebra $(\Grk,\wedge_k,*)$ on the \Hs\ $\h_k$.
Let $M_k$ be the bicommutant of $\pi_k(\Grk)$.
The form $\tau_k:a\in M_k\longmapsto \langle a,1\rangle$ is a faithful normal trace.
Consider the GNS \Hs\ coming from this state $L^2(M_k)$.
We identify the standard representation of $M_k$ on $L^2(M_k)$ with the representation $\pi_k$ extended to $M_k$.
The algebra $M_k$ is unitaly embedded in $M_{k+1}$ via the map
$$\includegraphics[width=5cm]{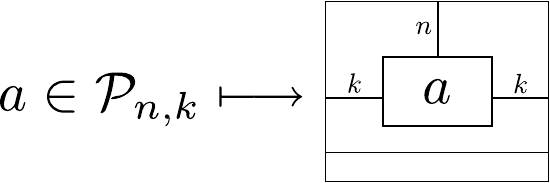}.$$
Hence, we have a tower of finite \vna s $M_0\subset M_1\subset\cdots\subset M_k\subset\cdots$.
We call it \textit{the tower associated to $\Pl$}.
The \vna\ $M_0$ corresponds to the \vna\ $\MP$ of the precedents sections.
Consider the cup subalgebra $A\subset M_k$.

\begin{theo}\label{theo_bimod_LMk}
The $A$-bimodule $L^2(M_k)$ is isomorphic to 
$$(\Pl_{0,k}\otimes L^2(A))\oplus\bigoplus_{j=0}^\infty L^2(A)\otimes L^2(A).$$
\end{theo}

\begin{proof}
We fix an integer $n\geqslant 0$.
Consider the \vesp\
$$\includegraphics[width=10cm]{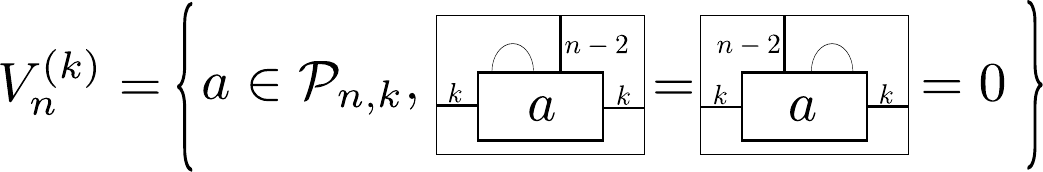}$$
Let $V^{(k)}$ be the orthogonal direct sum of the $V_n^{(k)}$.
We denote by $_A{V^{(k)}}_A$ (resp. $_A{\Pl_{1,k}}_A$, resp. $_A{\Pl_{0,k}}_A$) the $A$-bimodule generated by $V_n^{(k)}$ (resp. $\Pl_{1,k}$, resp. $\Pl_{0,k}$).
A similar proof of proposition \ref{prop_GJSW_LM_sum_V_P_A} shows that $\LMk$ is the orthogonal direct sum of those three bimodules.
Following the proof of \cite[Theorem 4.9]{JSW_orthogonal_approach_planar_algebra}, we get that the bimodule $_{A}{V^{(k)}}_{A}$ is isomorphic to an infinite direct sum of coarse correspondences.
Consider the map 
\begin{align*}
Z^k:\Pl_{1,k}&\longrightarrow {_A{\Pl_{1,k}}_A}\\
a&\longmapsto Z^k(a),
\end{align*}
where $$Z^k(a)=\frac{\cup\bullet a-\dmi a\bullet \cup}{\sqrt{\delta-\dmi}}.$$
With a similar proof of proposition \ref{prop_GJSW_ope_unit_equiv}, we have that the $A$-bimodule generated by any non null $a\in\Pl_{1,k}$ is isomorphic to the coarse correspondence.
Therefore, the $A$-bimodule generated by ${\Pl_{1,k}}$ is isomorphic to a sum of coarse correspondences.
Consider the bimodule $ _A{\Pl_{0,k}}_A$.
It is clearly isomorphic to $\Pl_{0,k}\otimes \LA$.
Therefore, the $A$-bimodule $\LMk$ is isomorphic to $(\Pl_{0,k}\otimes L^2(A))\oplus\bigoplus_{j=0}^\infty L^2(A)\otimes L^2(A).$
\end{proof}

\begin{cor}
For any $k\geqslant 0$, the \vna\ $M_k$ is a II$_1$ factor.
The \sfplalg\ of the subfactor $M_0\subset M_1$ is equal to the \unplal\ $\tilde\Pl=\{\Pl_{2n},\ n>0\}$.
\end{cor}

\begin{proof}
The $A$-bimodule structure of $\LMk$ given in the last theorem tells us that the relative comutant $M_k\cap M_0'$ is equal to the algebra $(\Pl_{2k},\times)$.
Then we follow the strategy of  \cite{JSW_orthogonal_approach_planar_algebra} to show that each $M_k$ is a factor and that the planar algebra of the subfactor $M_0\subset M_1$ is equal to $\tilde\Pl$.
\end{proof}

\begin{rem}
Consider the tower $\{\tilde M_k,\ k\geqslant 0\}$ associated to $\tilde\Pl$.
This provides a countable family of subfactors: $\tilde M_k\subset M_k$ coming from the inclusion 
$$\bigoplus_n \Pl_{2n,k}\subset \bigoplus_m \Pl_{m,k}.$$
Note that $\tilde M_k$ is the fixed point algebra of the involution $a\in \Pl_{n,k}\mapsto (-1)^n a$.
Therefore, the Jones index of $\tilde M_k\subset M_k$ is equal to $2$.
\end{rem}

%%%%%%%%%%%%%%%%%%%%%%%%%%%%%%%%%%%%%%%%%%%%%%%%%%%%%%%%%%%%%%%%%%%%%%%%%%%%%%%%%%%%%%%%%%%%%%%%%%%%%%%%%%%%%%%%%%%%%%%%%%%%%%%%%%%%%%%%%%%%%%%%%%%%%

\subsection{Maximal abelian subalgebras}

%%%%%%%%%%%%%%%%%%%%%%%%%%%%%%%%%%%%%%%%%%%%%%%%%%%%%%%%%%%%%%%%%%%%%%%%%%%%%%%%%%%%%%%%%%%%%%%%%%%%%%%%%%%%%%%%%%%%%%%%%%%%%%%%%%%%%%%%%%%%%%%%%%%%

In this appendix, we review some definitions about maximal abelian subalgebras (MASAs).
We present some invariants for those objects.
We prove a theorem that generalizes the proof of Theorem \ref{theo_GJSW_main_theo} for diffuse abelian subalgebra of a finite \vna.
Then, we discuss a conjecture of MASAs in the free group factor that is related to the cup subalgebra.
Note that we stop using the symbol $\star$ for the multiplication in a \vna.

\begin{defi}

\textit{Group normalizer:}\\
We define the group normalizer $\NMA$ that is the group of unitaries $u\in M$ such that $uAu^*=A$.
If the \vna\, generated by $\NMA$ is equal to $A$, the MASA is called singular.
If the \vna\, generated by $\NMA$ is equal to $M$, the MASA is called regular.

Suppose $M$ is a II$_1$ factor with its unique faithful trace $\tau$.
Let $\LM$ be the GNS Hilbert space associate to the trace $\tau$.

\textit{Pukanszky invariant:}\\
Consider the abelian \vna\, $\mathcal A$ equal to the bicommutant
$$\{\pi(A),\rho(A)\}''\subset\Bo(\LM),$$
where $\pi,\rho$ are the left and right action of $M$ on the \Hs\ $\LM$.
Let $P$ be the commutant of $\mathcal A$ acting on the orthogonal complement $\LM\ominus \LA$.
The algebra $P$ is a finite type I \vna.
Consider the subset of $n\in\N\cup\{\infty\}$ such that there exists a non null direct summand of type I$_n$ in $P$.
This set is the Pukanszky invariant of the MASA $\AM$, we denote it by $Puk(\AM)$.

\textit{Takesaki invariant:}\\
Let $\YDn$ be a \sps\ such that $A$ is isomorphic to the \vna\, of bounded measurable complex valued functions $L^\infty(Y,\nu)$.
Let $\pi,\rho$ be the left and right action of $M$ on the Hilbert space $\LM$, i.e. $\pi(x)\rho(y)(z)=xzy$, where $x,y,z\in M$.
Consider a measurable field of \Hs s $\{\h_t,\,t\in Y\}$ such that $\LM$ is equal to the direct integral
$$\intI \h_t d\nu(t),$$
such that $\rho(A)$ becomes the algebra of all diagonalizable operators.
Let $B\subset M$ be a separable $C^*$-subalgebra that is dense for the ultraweak topology.
Consider a measurable field of representations of $B$, $\{\pi_t,\,t\in Y\}$, such that
$$\pi\vert_B=\intI \pi_t d\nu(t),$$
where $\pi\vert_B$ denotes the restriction to $B$ of the standard representation.
Let $\mathcal R$ be the equivalence relation on $Y$ such that $(s,t)\in\mathcal R$ if and only if the representation $\pi_s$ is unitarily equivalent to the representation $\pi_t$.
Now let $\mathcal R,\mathcal R'$ be two equivalence relations on $Y$.
We define an equivalence relation "$\equiv$" such that $\mathcal R\equiv \mathcal R'$ if and only if there exists a Borel null set $N\subset Y$ such that $\mathcal R\backslash N^2=\mathcal R'\backslash N^2$.
Let $\widehat{\mathcal R}$ be the equivalence class of $\mathcal R$ for "$\equiv$".
It is the Takesaki invariant.
We say that a MASA is Takesaki simple if $\widehat{\mathcal R}$ is equal to the equivalence class of the trivial equivalence relation.
\end{defi}

Those definitions have been introduced by Dixmier \cite{Dixmier_anneaux_max_ab}, Pukanszky \cite{pukanszky_invariant} and Takesaki \cite{takesaki_invariant_masa}.

\begin{prop}\label{prop_GJSW_coarse MASA sing}
Consider a finite von Neumann algebra $M$ with a faithful trace $\tau$ and $A\subset M$ a diffuse abelian von Neumann subalgebra.
Let $\LM$ and $L^2(A)$ be the Hilbert spaces of the GNS construction associated to $\tau$.
Suppose there exists a Hilbert space $\mathcal K$ such that we have an isomorphism of bimodules:
$$_AL^2(M)_A\simeq{\ALAA}\oplus({_AL^2(A)\otimes \mathcal K\otimes L^2(A)_A}).$$
Then $M$ is a II$_1$ factor, $A$ is maximal abelian, singular, Takesaki simple, with Pukanszky invariant equal to the singleton $\{\dim \mathcal K\}$.
\end{prop}

\begin{proof}
Consider $\AM$ and $\mathcal K$ as in the hypothesis of the proposition.
Let us show that $A\subset M$ is maximal abelian.
Let $\YDn$ be a \sps\ such that $A$ is isomorphic to the \vna\ $L^\infty(Y,\nu)$. 
We have supposed that $A$ is diffuse, hence $\nu$ is non-atomic.
We denote by $\psi:A\longrightarrow L^\infty(Y,\nu)$ an isomorphisms of \vna s.
It induces an isomorphisms of Hilbert spaces $\overline{\Psi}:L^2(A)\longrightarrow L^2(Y,\nu)$.
Consider the Hilbert space $L^2(Y^2,\mathcal K,\nu\otimes \nu)$ of Borel measurable functions $f:Y^2\longrightarrow \mathcal K$ that are square integrable for the product measure $\nu\otimes \nu$, i.e.
$$\int_{Y^2}\Vert f(s,t)\Vert_{\mathcal K}^2 d(\nu\otimes\nu)(s,t)<\infty.$$
We have a isomorphisms of Hilbert spaces defined as follows:
\begin{align*}
\Ld\otimes \mathcal K\otimes \Ld & \longrightarrow L^2(Y^2,\mathcal K,\nu\otimes \nu)\\
g_1\otimes \xi\otimes g_2&\longmapsto ((s,t)\in Y^2\mapsto g_1(s)g_2(t)\xi).
\end{align*}
The Hilbert space $L^2(Y^2,\mathcal K,\nu\otimes \nu)$ as a structure of bimodule over the \vna\, $\Li$ that is:
$$(f_1.h.f_2)(s,t)=f_1(s)f_2(t)h(s,t),$$
for any $f_1,f_2\in\Li$ and any $h\in L^2(Y^2,\mathcal K,\nu\otimes \nu)$.
Consider the Hilbert space $\Ld$, it has a structure of bimodule that is:
$$(f_1.h.f_2)(t)=f_1(t)f_2(t)h(t),$$
for any $f_1,f_2\in\Li$ and any $h\in L^2(Y,\nu)$.
We suppose those Hilbert spaces equipped with the bimodule structures that we just described.
We have an isomorphism of bimodules:
\begin{equation}\label{LM_bimodule_decompo}
\varphi:{_AL^2(M)_A}\longrightarrow L^2(Y,\nu)\oplus L^2(Y^2,\mathcal K,\nu\otimes \nu),
\end{equation}
where $A$ is identified with $\Li$ and $L^2(A)$ with $\Ld$.\\
Let $\pi$ be the standard representation of $M$ on $L^2(M)$ and $\rho$ the standard representation of the opposite algebra $M^{op}$ on $L^2(M)$.
If we identify the two bimodules $_AL^2(M)_A$ and
$$L^2(Y,\nu)\oplus L^2(Y^2,\mathcal K,\nu\otimes \nu)$$
we have that
$$\pi(f)(h\oplus k)(s,t)=(f(t)h(t))\oplus (f(s)h(s,t))$$
and
$$\rho(f)(h\oplus k)(s,t)=(f(t)h(t))\oplus(f(t)h(s,t)),$$
where $f\in \Li$, $h\in \Ld$ and $k\in L^2(Y^2,\mathcal K,\nu\otimes \nu)$.

Let us show that $\AM$ is maximal abelian:

Consider an element $x\in M$ that commutes with $A$ such that its conditional expectation onto $A$ is equal to zero,
i.e. $E_A(x)=0$.
Let us identify the \vna\ $M$ and its image in the \Hs\ $\LM$.
The vector $\varphi(x)$ is orthogonal to $\Ld$ because $E_A(x)=0$, so there exists a vector $g\in L^2(Y^2,\mathcal K,\nu\otimes \nu)$ such that $\varphi(x)=g$.
Let $f\in\Ld$ be an injective function, we have that $\pi(f)(g)=\rho(f)(g)$ because $x$ commutes with $A$.
Hence, $f(s)g(s,t)=f(t)g(s,t)$ almost everywhere for the product measure $\nu\otimes \nu$.

This implies that $g$ is supported on the diagonal $\Delta Y=\{(t,t),\,t\in Y\}$.
The measure $\nu$ is non-atomic; thus, $(\nu\otimes\nu)(\Delta Y)=0$.
Therefore, $g=0$, hence $x=0$.
It means that $A'\cap M$ is included in $A$, so $\AM$ is maximal abelian.

Let us show that $M$ is a factor.
Consider a central element $x\in M\cap M'$, we have that $x\in M\cap A'=A$.
The element $x$ commutes with $M$; thus, for any vector $\eta\in\LM$,
$\pi(x)(\eta)=\rho(x)(\eta)$.
We denote the identity of the algebra $M$ by $1$.
Let $1\otimes \xi\otimes 1\in \LA\otimes \mathcal K\otimes \LA$, we have that $\pi(x)(1\otimes \xi\otimes 1)=x\otimes \xi\otimes 1$ and $\rho(x)(1\otimes \xi\otimes 1)=1\otimes \xi\otimes x$.
By identification we get that $x\in\C 1$, so $M$ is a factor.

The equality \ref{LM_bimodule_decompo} shows that the bimodule $\LM\ominus\LA$ is isomorphic to $L^2(Y^2,\mathcal K,\nu\otimes \nu)$.
Thus, it is the direct integral of measurable fields of Hilbert spaces $\{\mathcal K_{s,t},\,(s,t)\in Y^2\}$ over the probability space $(Y^2,\nu\otimes\nu)$, where for any $(s,t)\in Y^2$, $\mathcal K_{s,t}=\mathcal K$.
The Pukanszky invariant is, by definition, the essential value of the dimension function $d(s,t)=\dim\mathcal K_{s,t}$.
In our case, it is clearly equal to the singleton $\{\dim\mathcal K\}$.

Let us prove that $\AM$ is Takesaki simple \cite{takesaki_invariant_masa}.

Let $B$ be a separable, ultraweakly dense, $C^*$-subalgebra of $M$.
Consider the abelian $C^*$-algebra of continuous complex valued function $\mathcal C(Y)$, view as a subalgebra of $A$.
We suppose that $\mathcal C(Y)\subset B$.
We begin by diagonalizing the abelian \vna\, $\rho(A)$.
Let $\h_0$ be the Hilbert space equal to the orthogonal direct sum $\C\oplus L^2(Y,\mathcal K,\nu)$,
where $L^2(Y,\mathcal K,\nu)$ is the Hilbert space of measurable functions $g:Y\longrightarrow \mathcal K$ such that
$$\int_Y \Vert g(s)\Vert_{\mathcal K}^2 d\nu(s)<\infty.$$
Consider the tensor product of Hilbert spaces $\h:=\h_0\otimes \Ld$.
We have an isomorphism
$$\phi:\h_0\otimes \Ld\longrightarrow \Ld\oplus L^2(Y^2,\mathcal K,\nu\otimes \nu)$$
such that
$$\phi((z\oplus f)\otimes g)(s,t)=(g(t)z)\oplus (f(s)g(t)),$$
where $z\in \C$, $f\in L^2(Y,\mathcal K,\nu)$, $g\in \Ld$.
The isomorphism $\phi$ conjugates the right action of $A$, $\rho(A)$, and the diagonal algebra.
If $\xi\otimes g\in \h_0\otimes \Ld$ and $f\in\Li$, then
$$(\phi^*\rho(f)\phi)(\xi\otimes g)(t)=\xi\otimes f(t)g(t).$$
Let $\pi\vert_B$ be the restriction to $B$ of the standard representation.
We have that $\pi(B)$ commutes with the diagonal algebra $\rho(A)$, hence there exists a measurable field $\{\pi_t,\,t\in Y\}$ of representations unique almost everywhere such that
$$\pi\vert_B=\int^\oplus_Y \pi_td\nu(t).$$
We want to prove that $\AM$ is Takesaki simple.
We need to show that there exists a Borel null set $N\subset Y$ such that for any $s,t\in Y\backslash N$, $\pi_s$ is unitarily equivalent to
$\pi_t$ if and only if $s=t$.
Here, we denote the set $\{t\in Y,\,t\notin N\}$ by $Y\backslash N$.

Consider an \textit{injective} continuous function $f\in \mathcal C(Y)$.
Fix a $t_0\in Y$, and consider the operator $f_{t_0}$ acting on the Hilbert space $\h_0$ as follows:
$$f_{t_0}(z\oplus g)(s)=(f(t_0)z)\oplus (f(s)g(s)),$$
for any $z\in \C$, $g\in L^2(Y,\mathcal K,\nu)$ and $s\in Y$.
The collection of operators $\{f_t,\,t\in Y\}$ is a bounded measurable operator field.
Consider the decomposable operator
$$D_f:=\int^\oplus_Y f_td\nu(t)$$
acting on $\h_0\otimes \Ld$.
We clearly have that $D_f=\phi\pi(f)\phi^*$.

Let us show that $f_{t_1}$ is unitarily conjugate to $f_{t_2}$ if and only if $t_1=t_2$.
To do this, we prove that the operator $f_t$ has a unique eigenvalue equal to $f(t)$.
Consider the vector
$$z\oplus 0\in \C\oplus L^2(Y,\mathcal K,\nu),$$
where $z$ is a complex number different from zero.
We have that $f_t(z\oplus 0)=f(t)z\oplus 0$, so $f(t)$ is an eigenvalue of the operator $f_t$.
Let us show this is the only eigenvalue of $f_t$, to do this it is sufficient to show that the restriction of $f_t$ to the Hilbert space $L^2(Y,\mathcal K,\nu)$ does not have any eigenvalue.
Let $z\in \C$ and $h\in L^2(Y,\mathcal K,\nu)$ such that $f_t(h)=zh.$
We have that $f(s)h(s)=zh(s)$ almost everywhere.
The function $f$ is injective, so $f^{-1}(\{z\})$ is empty or is a singleton.
The measure $\nu$ is non-atomic; thus, $\nu(f^{-1}(\{z\}))=0$.
Therefore, $h(s)=0$ almost everywhere, hence $h=0$.
We have proved that the set of eigenvalues of the operator $f_t$ is equal to the singleton $\{f(t)\}$.
The set of eigenvalues of an operator is invariant under unitary conjugacy.
Hence by injectivity of the function $f$, we get that $f_{t_1}$ and $f_{t_2}$ are unitarily equivalent if and only if $t_1=t_2$.

By uniqueness of the decomposition of an operator, there exists a Borel null set $N\subset Y$ such that $\pi_t(f)=f_t$ for any $t\in Y\backslash N$.
Let $s,t\in Y\backslash N$, and suppose that the two representation $\pi_s$ and $\pi_t$ are unitarily equivalent.
In particular, $\pi_s(f)$ and $\pi_t(f)$ are unitarily equivalent; thus, $f_t$ and $f_s$ are unitarily equivalent.
This implies that $s=t$. We have proved that $\AM$ is Takesaki simple.

A simple MASA is singular by \cite{takesaki_invariant_masa}[Theorem 4.1].
\end{proof}

\begin{rem}
By a result of Popa \cite{Popa_Notes_Cartan}, if $\AM$ is a MASA such that $1$ is not in the Pukanszki invariant we have that $\AM$ is singular.
Hence, in the case where $\dim\mathcal K>1$ we have a simpler proof of this proposition.
\end{rem}

\begin{rem}
Consider the radial and the generator MASAs in the free group factor $L(\mathbb F_2)$ with to generators $a,b$.
It is know that those two MASAs are singular, with Pukanzsky invariant equal to the singleton $\{\infty\}$.
Using proposition \ref{prop_GJSW_coarse MASA sing}, one can show that the cup subalgebra has the same invariants.

If we take the subfactor planar algebra of non commutative polynomials of two variables with monomials of even degree, then the cup subalgebra is isomorphic to the generator MASA.
Furthermore, let us take the unshaded planar algebra $\Pl$ of all non commutative polynomials of two variables.
The \vna\, $\MP$ is isomorphic to the free group factor with two generators $L(\mathbb F_2)$.
The cup subalgebra has a very similar form as the radial MASA.
One question is if this cup subalgebra is isomorphic to the radial MASA?

Let $f$ be a real valued measurable bounded function on $\C$. 
Consider the self-adjoint operator $h_f:=f(a)+f(b)\in L(\mathbb F_2)$ obtained by functional calculus.
In which case $h_f$ generates a MASA of $A_f\subset L(\mathbb F_2)$ and which condition on $f$ could assure that the MASA $A_f\subset L(\mathbb F_2)$ is isomorphic to the radial MASA?
Note that the cup subalgebra is of this form when $\Pl$ is the unshaded planar algebra of non commutative polynomials.
\end{rem}

%%%%%%%%%%%%%%%%%%%%%%%%%%%%%%%%%%%%%%%%%%%%%%%%%%%%%%%%%%%%%%%%%%%%%%%%%%%%%%%%%%%%%%%%%%%%%%%%%%%%%%%%%%%%%%%%%%%%%%%%%%%%%%%%%%%%%%%%%%%%%%%%%
\textit{Acknowledgements:}

I would like to thank my advisor Vaughan Jones who proposed me this problem at the very beginning of my PhD.
I thank Mike Hartglass for helping me to write in English.

%%%%%%%%%%%%%%%%%%%%%%%%%%%%%%%%%%%%%%%%%%%%%%%%%%%%%%%%%%%%%%%%%%%%%%%%%%%%%%%%%%%%%%%%%%%%%%%%%%%%%%%%%%%%%%%%%%%%%%%%%%%%%%%%%%%%%%%%%%%%%%

\bibliographystyle{plain}

\begin{thebibliography}{99}

\bibitem{Dixmier_anneaux_max_ab}
\textsc{J. Dixmier},
Sous-anneaux abéliens maximaux dans les facteurs de type fini. Ann. Math., 59, (1954).

\bibitem{Dixmier_livre_hilbert}
\textsc{J. Dixmier}
Les algèbres d'opérateurs dans l'espace hilbertien, Paris, (1969).

\bibitem{Dykema_LF_t}
\textsc{K.J. Dykema}
Interpolated free group factors, Pacific J. Math., 163, (1994), 123-135.

\bibitem{GJS_random_matrices_free_proba_planar_algebra_and_subfactor}
\textsc{A. Guionnet, V.F.R. Jones and D. Shlyakhtenko}
Random matrices, free probability, planar algebras and subfactors, in E. Blanchard, D. Ellwood, M. Marcolli, H. Moscovici, S. Popa (EDs.), A Quanta of Maths: Non-commutative Geometry Conference in Honor of Alain Connes, in: Clay Math. Proc., 11, (2011), 201-240.

\bibitem{GJS_semifinite_algebra}
\textsc{A. Guionnet, V.F.R. Jones and D. Shlyakhtenko}
A semi-finite algebra associated to a planar algebra, J. Funct. Anal., 261, (2011), 1345-1360.

\bibitem{Jones_index_for_subfactors}
\textsc{V.F.R. Jones}
Index for subfactors, Invent. Math., 72, (1983), 1-25.

\bibitem{jones_planar_algebra}
\textsc{V.F.R. Jones}
Planar algebras I, ArXiv:9909027, (1999).

\bibitem{JSW_orthogonal_approach_planar_algebra}
\textsc{V.F.R. Jones, D. Shlyakhtenko and K. Walker}
An orthogonal approach to the subfactor of a planar algebra, Pacific J. Math., 246, (2010), 187-197.

\bibitem{Kato_perturbation_theory_lin_ope}
\textsc{T. Kato}
Perturbation theory for linear operators, New York, (1966).


\bibitem{sunder_constructionGJS}
\textsc{V. Kodiyalam and V.S. Sunder}
From subfactor planar algebras to subfactors, Internat. J. Math., 20, 10, (2009), 1207-1231.

\bibitem{Sunder_surlaconstructionGJSII}
\textsc{V. Kodiyalam and V.S. Sunder}
On the Guionnet-Jones-Shlyakhtenko construction for graphs, J. Funct. Anal., 260, (2011), 2635-2673.

\bibitem{Voiculescu_dykema_nica_Free_random_variables}
\textsc{A. Nica, K.J. Dykema and D.V. Voiculescu}
Free fandom variables, CRM, (1992).

\bibitem{peters_planar_haagerup_graph}
\textsc{E. Peters}
A planar algebra construction of the Haagerup subfactor, Internat. J. Math., 21, 8, (2010), 987-1045.

\bibitem{Popa_Notes_Cartan}
\textsc{S. Popa}
Notes on Cartan subalgebras in type II$_1$ factors, Math. Scand., 57, 1, (1985), 171-188.

\bibitem{popa_system_construction_subfactor}
\textsc{S. Popa}
An axiomatization of the lattice of higher relative commutants of a subfactor, Invent. Math., 120, 3, (1995), 427-445.

\bibitem{Popa_Shlyakhtenko_univ_prop_LF_subfactor}
\textsc{S. Popa and D. Shlyakhtenko}
Universal properties of $L(\mathbb F_\infty)$ in subfactor theory, Acta. Math., 191, (2003), 225-257.

\bibitem{pukanszky_invariant}
\textsc{L. Pukanszky}
On maximal abelian subrings of factor of type II$_1$, Can. J. Math., 12, (1960), 289-296.

\bibitem{Radulescu_Random_matrices_LF_t}
\textsc{F. Radulescu}
Random matrices, amalgameted free products and subfactors of the von Neumann algebra of a free group, of noninteger index, Invent. Math., 115, (1994), 347-389.

\bibitem{takesaki_invariant_masa}
\textsc{M. Takesaki}
On the unitary equivalence among the components of decompositions of representations of involutive Banach algebras and the associated diagonal algebras, Tohoku Math. J., II, Ser., 15, (1963), 365-393.

\bibitem{Voiculescu_sym_free_prod_C_alg_foncteur}
\textsc{D.V. Voiculescu}
Symmetries of some reduced free product $C^*$-algebras, in lecture notes in mathematics, Operator algebras and their connections with topology and ergodic theory, 1132, (1985), 556-588.




\end{thebibliography}

\end{document}